\documentclass{article}
\usepackage[a4paper, total={6.5in, 8.5in}]{geometry}

\usepackage{graphicx}       

\usepackage[latin1]{inputenc}
\usepackage[T1]{fontenc}
\usepackage{microtype} 

\usepackage{cite}
\usepackage{array,colortbl}
\usepackage{esvect}

\usepackage{amsmath,amsfonts,amssymb,bm,amsthm,mathtools,mathrsfs}

\DeclareFontFamily{U}{mathx}{\hyphenchar\font45}
\DeclareFontShape{U}{mathx}{m}{n}{
      <5> <6> <7> <8> <9> <10>
      <10.95> <12> <14.4> <17.28> <20.74> <24.88>
      mathx10
      }{}
\DeclareSymbolFont{mathx}{U}{mathx}{m}{n}
\DeclareFontSubstitution{U}{mathx}{m}{n}
\DeclareMathAccent{\widecheck}{0}{mathx}{"71}

\def\citep#1#2{\cite[{#1}]{#2}}

\theoremstyle{plain}
 \newtheorem{theorem}{Theorem}

\theoremstyle{definition}

\theoremstyle{remark}
 \newtheorem{remark}{Remark}

\usepackage{algorithm}
\usepackage{algorithmic}

\newcommand{\bbE}{{\mathbb{E}}}

\newcommand{\bbP}{{\mathbb{P}}}

\newcommand{\bbR}{{\mathbb{R}}}

\newcommand{\bbV}{{\mathbb{V}}}

\newcommand{\N}{{\mathbb{N}}} 
\newcommand{\R}{{\mathbb{R}}} 

\DeclareSymbolFont{bbold}{U}{bbold}{m}{n}
\DeclareSymbolFontAlphabet{\mathbbold}{bbold}
\newcommand{\ind}{{\mathbbold{1}}}

\newcommand{\calA}{{\mathcal{A}}}
\newcommand{\calB}{{\mathcal{B}}}
\newcommand{\calC}{{\mathcal{C}}}
\newcommand{\calD}{{\mathcal{D}}}
\newcommand{\calE}{{\mathcal{E}}}

\newcommand{\calK}{{\mathcal{K}}}
\newcommand{\calL}{{\mathcal{L}}}

\newcommand{\calO}{{\mathcal{O}}}

\newcommand{\calQ}{{\mathcal{Q}}}

\newcommand{\calT}{{\mathcal{T}}}
\newcommand{\calU}{{\mathcal{U}}}

\providecommand{\argmin}{\operatorname*{argmin}}

\usepackage{type1cm}        
\usepackage{makeidx}         
\usepackage{graphicx}        
\usepackage{multicol}        
\usepackage[bottom]{footmisc}

\usepackage{url}

\usepackage{newtxtext}       %
\usepackage{authblk}

\makeindex             

\usepackage{fancyhdr}

\begin{document}
\pagestyle{fancy}
\fancyhead[L]{QCLMC for random elliptic PDEs}
\fancyhead[R]{C. A. Beschle and A. Barth}

\title{Quasi continuous level Monte Carlo for random elliptic PDEs}
\author[$\dagger$]{Cedric Aaron Beschle}
\author[$\star$]{Andrea Barth}

\affil[$\dagger$]{cedric.beschle@ians.uni-stuttgart.de}
\affil[$\star$]{andrea.barth@ians.uni-stuttgart.de}

\date{}

\maketitle

\abstract{
This paper provides a framework in which multilevel Monte Carlo and continuous level Monte Carlo can be compared. In continuous level Monte Carlo the level of refinement is determined by an exponentially distributed random variable, which therefore heavily influences the computational complexity. We propose in this paper a variant of the algorithm, where the exponentially distributed random variable is generated by a quasi Monte Carlo sequence, resulting in a significant variance reduction. In the examples presented the quasi continuous level Monte Carlo algorithm outperforms multilevel and continuous level Monte Carlo by a clear margin.
   }

\section{Introduction}
\label{BeBa_sec:introduction}
In the last decades multilevel Monte Carlo (MLMC) methods have been applied to a plethora of problems in stochastic modelling and uncertainty quantification (see e.g.~\cite{BeBa_Giles2015_MLMC, BeBa_Barth2011_Elliptic, BeBa_Cliffe2011_MLMCPDE, BeBa_Teckentrup2013_MLMCPDE, BeBa_Haji-Ali2016_MLMCOptimization}). The method relies on a hierarchy of approximations arranged as a telescoping sum, resulting in a variance reduction. In fact, in certain situations the multilevel estimator has asymptotically the same computational complexity as one solve of the deterministic problem on the finest discretization of the hierarchy. Having a preset finest discretization level the multilevel estimator is a biased estimator. The continuous level Monte Carlo (CLMC) estimator circumvents this issue by defining the estimator as  a stochastic process (see~\cite{BeBa_Detommaso2019_CLMC}). In the case of uniform mesh refinement it may be considered as the unbiased Rhee and Glynn estimator, introduced in \cite{BeBa_Rhee2015_unbiasedSDE} in the context of SDEs. Based on their idea, an unbiased multilevel Monte Carlo estimator for random elliptic PDEs is proposed in \cite{BeBa_Li2018_unbiasedMLMCPDE}. The advantage of the CLMC estimator is, that it naturally allows for sample adaptive mesh refinement. The level of refinement is here given by an exponentially distributed random variable, which in turn means that the computational complexity of a CLMC estimator relies heavily on the concrete samples of this random variable in the simulation. To reduce the variance in this sampling, we propose a quasi Monte Carlo variant of the algorithm, introduced together with a convergence proof in \cite{BeBa_Beschle2023_QCLMC}.  

Quasi-random numbers or quasi-random methods gained popularity in the last century and have applications in different kind of fields of numerical simulation. For detailed surveys of quasi-random methods, cf.~\cite{BeBa_Niederreiter1978_QMC} and \cite{BeBa_Niederreiter1992_QMC}.
Quasi-random numbers have also been applied to partial differential equations (PDEs) with random coefficient in several works before. In \cite{BeBa_Kuo2012_QMC}, a quasi Monte Carlo finite element method is applied to an elliptic PDE with random coefficient. This is extended to multilevel quasi Monte Carlo \cite{BeBa_Kuo2015_MLQMC} by the same authors. A quasi Monte Carlo method for an elliptic PDE with random coefficient is also considered in \cite{BeBa_Graham2011_QMC} and an extension of it to log-normal coefficients in \cite{BeBa_Graham2015_QMClognormal}. In these works, the term 'quasi' refers to the sampling of the random field, which is not what we consider here.
In our proposed variant of the method it refers to the maximal level of refinement per sample, which is the reason we term it quasi continuous level Monte Carlo (QCLMC) instead of continuous level quasi Monte Carlo.  A similar idea was mentioned in a Remark in \cite{BeBa_Vihola2018_unbiasedMLMC} in the general framework of unbiased MLMC estimators. 

To be able to compare MLMC and CLMC we restate complexity theorems for both methods (Section~\ref{BeBa_sec:multilevel_and_continuous_level_monte_carlo_method}) and introduce the QCLMC variant in Section~\ref{BeBa_sec:quasi_continuous_level_monte_carlo_method}. In Section~\ref{BeBa_sec:random_pde_model} we state a PDE model with a random discontinuous coefficient and a corresponding a-posteriori error estimation in Section~\ref{BeBa_sec:a_posteriori_error_estimation}. We show the performance of MLMC, CLMC and QCLMC in Section~\ref{BeBa_sec:numerical_experiments} in estimating the expectation of the random PDE. 
We treat two explicit examples for the random coefficient in this paper. They demonstrate the negative effect of the discontinuities on the regularity of the pathwise weak solution and on the pathwise convergence rate for standard numerical algorithms on standard meshes. CLMC should perform better than MLMC for such problems as solution samples have distinct areas where error contributions are high compared to other areas. However, the exponentially distributed maximum refinement renders it worse than MLMC when sampled by pseudo-random numbers. The variance reduction by a quasi-random sequence is essential for optimal computational complexity. 

\section{Multilevel and Continuous Level Monte Carlo method}
\label{BeBa_sec:multilevel_and_continuous_level_monte_carlo_method}
Let $\calQ$ denote a real valued quantity of interest of the solution to an underlying stochastic model. 
In the context of uncertainty quantification we are interested in estimating the mean value $\bbE[\calQ]$ of the quantity of interest up to some desired accuracy. Let $Q_L$ be an approximation of $\calQ$ by a discretization-based numerical scheme to some resolution parameter $L \in \N_0$, e.g., corresponding to the degrees of freedom (DOF) of a mesh. Under the assumption that $\bbE[Q_L] \to \bbE[\calQ]$ $\mathbb{P}\textup{-almost surely}$ for $L \to \infty$, our focus is on computing accurate estimates $\widehat{Q}_L^{est}$ to $\bbE[Q_L]$. 
The standard statistical method, the Monte Carlo method, uses an average of approximation samples at a desired resolution. Given $M \in \N$ independent approximation samples $(Q_L^{(k)})_{k=1}^M$ of $Q_L$ at resolution $L>0$, the mean value is estimated as
\begin{equation}
    \bbE[Q_L] \approx \widehat{Q}_L^{\text{MC}} := \frac{1}{M} \sum_{k=1}^M Q_L^{(k)}.
 \label{BeBa_eq:montecarlo}
\end{equation}
In this work we investigate two extensions of this method for the estimation of the mean value. On the one hand the multilevel Monte Carlo method (MLMC), cf.~\cite{BeBa_Giles2015_MLMC, BeBa_Cliffe2011_MLMCPDE}, and on the other hand the continuous level Monte Carlo method (CLMC) developed in \cite{BeBa_Detommaso2019_CLMC}.
The accuracy of the estimations is quantified by the mean-squared-error
\begin{equation}
    \text{MSE}:= \bbE\big[(\widehat{Q}_L^{est} - \bbE[\calQ])^2 \big] = \bbV[\widehat{Q}_L^{est}] +  \bbE[\widehat{Q}_L^{est} - \calQ]^2.
    \label{BeBa_eq:MSE}
\end{equation}
This expansion is the basis for the proofs of the MLMC, respectively CLMC complexity theorems, that are stated in Sections \ref{BeBa_subsec:MLMC} and \ref{BeBa_subsec:CLMC}, respectively. 
Since the MC estimator \eqref{BeBa_eq:montecarlo} is an unbiased estimator for $\bbE[Q_L]$, i.e., $\bbE[\widehat{Q}_L^\text{MC}] = \bbE[Q_L]$, the MSE \eqref{BeBa_eq:MSE} for MC reduces to
\begin{equation*}
 \bbE\big[(\widehat{Q}_L^\text{MC} - \bbE[\calQ])^2 \big] = \bbV[\widehat{Q}_L^\text{MC}] +  \bbE[Q_L - \calQ]^2,
\end{equation*}
consisting of the estimators variance and the squared bias of the approximation in the quantity of interest. 
\subsection{Multilevel Monte Carlo method}
\label{BeBa_subsec:MLMC}
MLMC extends MC by combining samples from different resolutions $(Q_\ell)_{\ell=0}^L$, referred to as levels, in a telescoping sum. We formulate the MLMC estimator for the difference quantity $\bbE[\calQ - Q_0]$ utilizing the linearity of the mean value 
$    \bbE[Q_L - Q_0] =  \ \sum_{\ell=1}^L \bbE[Q_\ell - Q_{\ell-1}]$,
and estimating each difference separately by MC averages \eqref{BeBa_eq:montecarlo}
\begin{equation*}
    \bbE[\calQ - Q_0] \approx \widehat{Q}_{0,L}^\text{MLMC}  := \sum_{\ell=1}^L \frac{1}{M_\ell} \sum_{k=1}^{M_\ell} Q_\ell^{(k)} - Q_{\ell-1}^{(k)}.
\end{equation*}
The differences $Q_\ell^{(k)} - Q_{\ell-1}^{(k)}$ on consecutive levels stemming from the same sample $k =1,...,M_\ell$ for $\ell=1,...,L$ are positively correlated, leading to a decrease in variance from the coarsest to the finest level. MLMC is an unbiased estimator for $\bbE[Q_L - Q_0]$, i.e., $\bbE[\widehat{Q}_{0,L}^\text{MLMC} - Q_0] = \bbE[Q_L - Q_0],$ so the MSE 
for MLMC reduces to
\begin{equation}
 \text{MSE}_{0,L}^\text{MLMC} = \sum_{\ell=1}^L\frac{1}{M_\ell} \bbV[Q_\ell - Q_{\ell-1}] +  \bbE[Q_L - \calQ]^2.
 \label{BeBa_eq:MSEMLMC}
\end{equation}
The following fundamental theorem, cf. \cite[Theorem~$2.1$]{BeBa_Giles2015_MLMC}, \cite[Theorem~$1$]{BeBa_Cliffe2011_MLMCPDE}, and \cite[Theorem~$2.5$]{BeBa_Teckentrup2013_MLMCPDE} is the convergence result for MLMC. Its proof is based on the MSE decomposition into a variance and a bias term \eqref{BeBa_eq:MSEMLMC}.
\begin{theorem}[MLMC - complexity theorem]
\label{BeBa_thm:MLMC-complexity}
    Assume there exists a factor $s \in \bbR$ with $N_\ell = s \,N_{\ell-1} = s^\ell N_0$ for the $\text{DOF}$ $N_\ell$ at level $\ell$. 
    Further, suppose there exist positive constants $\alpha_M, \beta_M, \gamma_M, c_1, c_2, c_3$ with $\min\{\beta_M, 2\alpha_M\}>\gamma_M$ such that for any $\ell \in \N$:
    \begin{subequations}
        \begin{equation}
                \bbE [Q_\ell - Q_{\ell-1}] \leq c_1 s^{-\alpha_M \ell}, 
            \label{BeBa_eq:MLMC_meandecay}
        \end{equation}   
        \begin{equation}
            \bbV [Q_\ell - Q_{\ell-1}] \leq  c_2 s^{-\beta_M \ell}, 
            \label{BeBa_eq:MLMC_variancedecay}
        \end{equation}   
        \begin{equation}
            \calC[Q_\ell - Q_{\ell-1}]\leq  c_3s^{\gamma_M \ell}.
            \label{BeBa_eq:MLMC_costincrease}
        \end{equation}   
    \end{subequations}
    Then, for any $\varepsilon \in (0,1)$, there exist $L \in \N_0$ and a sequence $(M_\ell)_{\ell=1}^{L}$ such that 
    \begin{equation*}
        \text{MSE}_{0,L}^\text{MLMC} \leq \varepsilon^2 \quad\hbox{ and }\quad \calC[\widehat{Q}_{0,L}^\text{MLMC}] \leq C \varepsilon^{-2},
    \end{equation*}
    where $C>0$ is a constant independent of $\varepsilon$.
\end{theorem}
\begin{proof}
 The proof is based on the MSE expansion \eqref{BeBa_eq:MSEMLMC} with the aim of bounding the MSE by $\varepsilon^2$. This is achieved by splitting the error contribution equally into the variance and the bias term.
In order to obtain a MLMC algorithm with optimizable cost, we introduce a weighting factor $b_w \in (0,1)$, similar to \cite{BeBa_Haji-Ali2016_MLMCOptimization} to obtain weighted error contributions
\begin{equation}
 \text{MSE}_{0,L}^\text{MLMC} = \sum_{\ell=1}^L\frac{1}{M_\ell} \bbV[Q_\ell - Q_{\ell-1}] +  \bbE[Q_L - \calQ]^2 \overset{!}{\leq} (1-b_w) \varepsilon^2 + b_w \, \varepsilon^2 = \varepsilon^2.
 \label{BeBa_eq:bias_weight_MLMC}
\end{equation}
 To bound the bias we use \eqref{BeBa_eq:MLMC_meandecay} and the geometric series for $s^{-\alpha_M} < 1$:
\begin{equation*}
 \begin{aligned}
   \ |\bbE[\calQ - Q_L]| &\leq \sum_{\ell=L+1}^\infty |\bbE[Q_\ell - Q_{\ell-1}]| = \sum_{\ell=L+1}^\infty c_1  s^{-\alpha_M \ell} \\
  = & \ c_1  s^{-\alpha_M L} s^{-\alpha_M} \sum_{\ell=0}^\infty s^{-\alpha_M \ell} = c_1  s^{\alpha_M L} \frac{s^{-\alpha_M}}{1- s^{-\alpha_M}} = \frac{c_1 s^{\alpha_M L}}{s^{\alpha_M}-1}.
 \end{aligned}
\end{equation*}
To obtain a bias smaller than $\sqrt{b_w} \,\varepsilon$ we choose
\begin{equation}
    \frac{c_1 \, s^{\alpha_M L}}{s^{\alpha_M}-1} \leq \sqrt{b_w} \, \varepsilon \Leftrightarrow L = \bigg\lceil \frac{1}{\alpha_M} \log_s \bigg(\frac{c_1 }{\sqrt{b_w}\, \varepsilon (s^{\alpha_M}-1)} \bigg) \bigg\rceil. 
    \label{BeBa_eq:bias_bound}
\end{equation}
It remains to bound the variance. Fixing the computational cost to $\calC_{fix}>0$, we minimize the variance by treating the sample number $M_\ell$ as a continuous variable. With the total cost of the estimator given by 
 $\calC[\widehat{Q}_L^\text{MLMC}] = \sum_{\ell=1}^L M_\ell \,\calC[Q_\ell - Q_{\ell-1}]$,
and the total variance given by
 $\bbV[Q_L^\text{MLMC}] = \sum_{\ell=1}^L\frac{1}{M_\ell} \bbV[Q_\ell - Q_{\ell-1}]$,
we set up the Lagrangian for the minimization:
\begin{equation*}
 \calL(M_0,...,M_L, \lambda) := \sum_{\ell=1}^L\frac{1}{M_\ell} \bbV[Q_\ell - Q_{\ell-1}] + \lambda \bigg(\sum_{\ell=1}^L M_\ell \, \calC[Q_\ell - Q_{\ell-1}] - \calC_{fix} \bigg),
\end{equation*}
where $\lambda > 0$ denotes the Lagrange multiplier.
Differentiation with respect to $M_\ell$ leads to 
\begin{equation*}
 \frac{d\calL}{d M_\ell} = - \frac{\bbV[Q_\ell - Q_{\ell-1}]}{M_\ell^2} + \lambda\, \calC[Q_\ell - Q_{\ell-1}] \overset{!}{=} 0.
\end{equation*}
Rearranging and using \eqref{BeBa_eq:MLMC_variancedecay} and \eqref{BeBa_eq:MLMC_costincrease} considered as proportionalities yields the sample numbers
\begin{equation}
    M_\ell = \sqrt{\frac{1}{\lambda} \, \frac{\bbV[Q_\ell - Q_{\ell-1}]}{\calC[Q_\ell - Q_{\ell-1}]}} = \widetilde{\lambda} \sqrt{\frac{s^{-\beta_M\ell}}{s^{\gamma_M\ell}}} = \widetilde{\lambda}  s^{-\frac{\beta_M + \gamma_M}{2}\ell},
    \label{BeBa_eq:samplenumber-proportionality}
\end{equation}
where the introduced constant of proportionality $\widetilde{\lambda}$ is determined such that the overall variance is smaller than $(1-b_w)\varepsilon^2$. 
Using Equation \eqref{BeBa_eq:samplenumber-proportionality} and the bound \eqref{BeBa_eq:MLMC_variancedecay} we compute for the total variance
\begin{equation*}
 \begin{aligned}
     \ \bbV[\widehat{Q}_{0,L}^\text{MLMC}] &= \sum_{\ell=1}^{L} \frac{1}{M_\ell} \bbV[Q_\ell - Q_{\ell-1}] \leq \frac{1}{\widetilde{\lambda}} c_2 \sum_{\ell=1}^{L}  s^{\frac{(\beta_M + \gamma_M)}{2}\ell} s^{-\beta_M \ell} \\
    = & \ \frac{1}{\widetilde{\lambda}} c_2 \sum_{\ell=1}^{L} s^{\frac{\gamma_M - \beta_M}{2}\ell} = \frac{1}{\widetilde{\lambda}} c_2 \, s^{\frac{\gamma_M - \beta_M}{2}} \frac{1 - s^{\frac{\gamma_M - \beta_M}{2}L}}{1 - s^{\frac{\gamma_M - \beta_M}{2}}},
 \end{aligned}  
\end{equation*}
for $\beta_M > \gamma_M$ and thus $s^{\frac{\gamma_M - \beta_M}{2}} < 1$.
This in turn yields
\begin{equation*}
 \widetilde{\lambda} \geq \frac{c_2}{(1-b_w)\varepsilon^2} s^{\frac{\gamma_M - \beta_M}{2}} \frac{1 - s^{\frac{\gamma_M - \beta_M}{2}L}}{1 - s^{\frac{\gamma_M - \beta_M}{2}}},
\end{equation*}
and we finally obtain
\begin{equation}
     M_\ell = \Bigg\lceil \frac{c_2}{(1-b_w) \varepsilon^2} s^{\frac{\gamma_M - \beta_M}{2}} \frac{1 - s^{\frac{\gamma_M - \beta_M}{2}L}}{1 - s^{\frac{\gamma_M - \beta_M}{2}}} s^{-\frac{\beta_M + \gamma_M}{2}\ell} \Bigg\rceil.
     \label{BeBa_eq:sample_sizes_MLMC}
\end{equation}
Bounding the Gauss bracket in $M_\ell$ by adding $1$, the cost of the overall estimator accumulates to
\begin{equation}
\begin{aligned}
    & \ \calC[\widehat{Q}_{0,L}^\text{MLMC}] = \sum_{\ell=1}^L M_\ell \,\calC[Q_\ell - Q_{\ell-1}] \leq \sum_{\ell=1}^L (1 + M_\ell) \,\calC[Q_\ell - Q_{\ell-1}]\\
    & \ \leq c_3 \sum_{\ell=1}^L s^{\gamma_M \ell}  + \frac{c_2 c_3}{(1-b_w) \varepsilon^2} s^{\frac{\gamma_M - \beta_M}{2}} \frac{1 - s^{\frac{\gamma_M - \beta_M}{2}L}}{1 - s^{\frac{\gamma_M - \beta_M}{2}}} \sum_{\ell=1}^L s^{-\frac{\beta_M + \gamma_M}{2}\ell} s^{\gamma_M \ell}\\
    & \ \leq c_3 s^{\gamma_M} \frac{s^{\gamma_M L} - 1}{s^{\gamma_M} - 1} + 
        \frac{\,c_2\, c_3 }{(1-b_w) \varepsilon^2}  s^{\gamma_M - \beta_M} \bigg(\frac{1 - s^{\frac{\gamma_M - \beta_M}{2}L}}{1 - s^{\frac{\gamma_M - \beta_M}{2}}}\bigg)^2 \\
    & \ \leq c_3  \frac{s^{\gamma_M}}{s^{\gamma_M} - 1} s^{\gamma_M L} + \varepsilon^{-2}
    \frac{\,c_2\, c_3 }{(1-b_w)}  s^{\gamma_M - \beta_M} \bigg(\frac{1}{1 - s^{\frac{\gamma_M - \beta_M}{2}}}\bigg)^2,
\end{aligned}
\label{BeBa_eq:MLMC_cost_bound}
\end{equation}
where we used the fact that $c_3 \frac{s^{\gamma_M}}{s^{\gamma_M} - 1} >0$ and $0 < 1- s^{\frac{\gamma_M - \beta_M}{2}L} < 1$ for all $L>0, s>1$ for $\beta_M > \gamma_M$. We further compute
\begin{equation*}
 \begin{aligned}
    s^{\gamma_M L} \leq s^{\gamma_M \big(\frac{1}{\alpha_M} \log_s \big(\frac{c_1 }{\sqrt{b_w}\, \varepsilon (s^{\alpha_M}-1)} \big) + 1\big)}& \ =  \varepsilon^{-\frac{\gamma_M}{\alpha_M}} \bigg(\frac{c_1 }{\sqrt{b_w} (s^{\alpha_M}-1)} \bigg)^\frac{\gamma_M}{\alpha_M} s^{\gamma_M} \\
    & \ \leq  \varepsilon^{-2} \bigg(\frac{c_1 }{\sqrt{b_w} (s^{\alpha_M}-1)} \bigg)^\frac{\gamma_M}{\alpha_M} s^{\gamma_M},
 \end{aligned}
\end{equation*}
for $2 \alpha_M > \gamma_M$.
With a constant $C:=C(c_1,c_2,c_3,\alpha_M,\beta_M,\gamma_M,s,b_w)>0$ we conclude
\begin{equation*}
  \begin{aligned}
    \calC[\widehat{Q}_{0,L}^\text{MLMC}] \leq & \ \varepsilon^{-2}  \Bigg(\bigg(\frac{c_1 }{\sqrt{b_w} (s^{\alpha_M}-1)} \bigg)^\frac{\gamma_M}{\alpha_M} \frac{s^{2 \gamma_M}}{s^{\gamma_M} - 1} + \frac{\,c_2\, c_3 }{(1-b_w)}   \bigg(\frac{s^{\frac{\gamma_M - \beta_M}{2}}}{1 - s^{\frac{\gamma_M - \beta_M}{2}}}\bigg)^2 \Bigg) \\
    = & \ C \, \varepsilon^{-2}.
    \end{aligned}
\end{equation*}
\,
\end{proof}
\begin{remark}
    The complexity theorem \ref{BeBa_thm:MLMC-complexity} is formulated for the case, where the expected cost to compute the correction samples $Q_\ell - Q_{\ell-1}$, grows slower than the variance of the corrections decreases $\beta_M > \gamma_M$ and at most twice as slow as the bias decreases $2 \alpha_M > \gamma_M$.
    In this case the asymptotic cost of MLMC is of the same order as the cost to compute a single solution sample of the random PDE. Thus, this is the best case for MLMC in terms of total cost to MSE ratio. In $d=1$ and with an optimized PDE solver in $d=2$, a value of $\gamma_M=1$ can be achieved. However, a standard linear solver in $d=2$ already admits $\gamma_M=1.5$ and in $d=3$ the cost grows even faster and one obtains a cost growth of $\varepsilon^{-\kappa}$ with $\kappa>2$ to achieve a MSE of $\varepsilon^2$ in MLMC.
\end{remark}
\subsection{Optimal MSE weight}
\label{BeBa_subsec:optimal_mse_weight}
The choice of the optimal weighting factor $b_w>0$ in the MSE expansion of MLMC \eqref{BeBa_eq:bias_weight_MLMC} for a practical application is very problem dependent. We refer to \cite[Chapter~$2$]{BeBa_Giles2015_MLMC} for a remark on the case $\beta_M > \gamma_M$, where the major computational cost lies on the coarsest level and the maximal level can be increased at comparably small cost. This translates into choosing $b_w>0$ close to $0$, but in a practical setting one is limited by the computational power of the device allowing only to compute up to some maximal level $0 < \bar{L} < \infty$. Therefore, we derive a numerically computable cost formula that depends on the estimated parameters from Section \ref{BeBa_subsec:MLMC_parameter_estimates} to find the optimal weighting in the MSE expansion of MLMC \eqref{BeBa_eq:bias_weight_MLMC} in terms of computational cost and a bound on the maximally computable level.
We start with a variation of the cost formula
\begin{equation}
 \calC[\widehat{Q}_{0,L}^\text{MLMC}] \leq \sum_{\ell=1}^L \max\{2, 1+M_\ell\} \,\calC[Q_\ell - Q_{\ell-1}],
\end{equation}
where the Gauss bracket in the sample size formula \eqref{BeBa_eq:sample_sizes_MLMC} is included by adding $1$. Further, we consider a minimum of $2$ samples on each level to include the minimal amount of samples needed to compute a variance estimate in the MLMC algorithm, not considering the quality of such an estimate here. 
We continue as in the proof of the cost bound \eqref{BeBa_eq:MLMC_cost_bound} in the MLMC complexity theorem, but only until the third inequality, to arrive at
\begin{equation}
\begin{aligned}
  & \ \calC[\widehat{Q}_{0,L}^\text{MLMC}] \leq  \max\bigg\{2 c_3 s^{\gamma_M} \frac{s^{\gamma_M L} - 1}{s^{\gamma_M} - 1}, \frac{\,c_2\, c_3 }{(1-b_w) \varepsilon^2}  s^{\gamma_M - \beta_M} \bigg(\frac{1 - s^{\frac{\gamma_M - \beta_M}{2}L}}{1 - s^{\frac{\gamma_M - \beta_M}{2}}}\bigg)^2\bigg\} \\ & \ + c_3 s^{\gamma_M} \frac{s^{\gamma_M L} - 1}{s^{\gamma_M} - 1} =: \widehat{\calC}^\text{MLMC}(\varepsilon,L,c_2,c_3,\beta_M,\gamma_M,b_w).
 \end{aligned}
 \label{BeBa_eq:MLMC_cost_computable}
\end{equation}
Treating the level $L$ from Equation \eqref{BeBa_eq:bias_bound} as a continuous variable and inserting it into Equation \eqref{BeBa_eq:MLMC_cost_computable} we obtain a cost formula that depends on the given tolerance $\varepsilon$, the estimated parameters $c_1$, $c_2$, $c_3$, $\alpha_M$, $\beta_M$, $\gamma_M$ from Section \ref{BeBa_subsec:MLMC_parameter_estimates} and the value $s$.
The total cost of the MLMC algorithm is minimized by determining
\begin{equation}
 \hat{b}_w = \argmin_{b_w \in (b_{w,low},\,1)} \widehat{\calC}^\text{MLMC}(\varepsilon,c_1,c_2,c_3,\alpha_M,\beta_M,\gamma_M,s,b_w),
 \label{BeBa_eq:MLMC_optimal_bias}
\end{equation}
 where the lower bound to the weighting factor $b_{w,low}$ is obtained by first minimizing
\begin{equation*}
  b_{w,low} = \argmin_{b_w \in (0,1)}|\bar{L} - L| = \argmin_{b_w \in (0,1)}\bigg|\bar{L} - \frac{1}{\alpha_M} \log_s \bigg(\frac{c_1 }{\sqrt{b_w}\, \varepsilon (s^{\alpha_M}-1)} \bigg) \bigg|.
\end{equation*}
The resulting $\hat{b}_w$ from Equation \eqref{BeBa_eq:MLMC_optimal_bias} minimizes the computational cost in MLMC necessary to achieve a desired tolerance, under consideration of the computational resources available in a practical simulation.
\subsection{MLMC on the fly algorithm}
\label{BeBa_subsec:mlmc_algorithm}
We formulate the MLMC algorithm \ref{BeBa_alg:MLMC} that we used in our numerical experiments in Section \ref{BeBa_subsec:method_comparison}. It is a modification of the algorithm of \cite[Section~$5$]{BeBa_Giles2015_MLMC}, that balances the contributions to the MSE \eqref{BeBa_eq:MSEMLMC} from the variance and bias on the fly while keeping the computational cost minimal. After determining the optimal weighting factor $\hat{b}_w$ for the MSE contributions by Equation \eqref{BeBa_eq:MLMC_optimal_bias} as described in Section \ref{BeBa_subsec:optimal_mse_weight}, the algorithm starts with a variance estimate $\bbV_L \approx \bbV[Q_L - Q_{L-1}]$ by an initial number of samples $M_{L,ini} \in \N$ on the first three levels $L=0,1,2$. These estimates are used to determine the optimal number of samples on each level, such that the overall variance is smaller than $(1-\hat{b}_w) \varepsilon^2$. The formula for this is
\begin{equation}
    M_{\ell} = \min\bigg\{2, \frac{1}{(1-\hat{b}_w)\varepsilon^2} \sqrt{\frac{\bbV_\ell}{N_\ell + N_{\ell - 1}}} \sum_{m=1}^L \sqrt{\frac{\bbV_m}{N_m + N_{m-1}}}\bigg\}.
    \label{BeBa_eq:opt-samples-MLMC}
\end{equation}
This is the formula from the Lagrangian minimization for the variance with fixed computational cost measured in degrees of freedom from the proof of the complexity theorem \ref{BeBa_thm:MLMC-complexity}. Additional levels are added until the bias of the estimator on the current finest level is smaller than $\sqrt{\hat{b}_w} \, \varepsilon$. An accurate estimation of the bias is non-trivial, but is motivated as follows:
\begin{equation*}
 \begin{aligned}
    \bbE[Q_\ell - Q_{\ell-1}] & \  = \bbE[Q_\ell - \calQ] - \bbE[Q_{\ell-1} - \calQ] \approx c_1 N_0^{-\alpha_M} \big(s^{-\alpha \ell} - s^{-\alpha_M(\ell-1)} \big) \\
    & \ = c_1 N_0^{-\alpha_M} s^{-\alpha_M \ell} \big(1 - s^{\alpha_M}\big) \approx \bbE[Q_\ell - \calQ](1-s^{\alpha_M}).
 \end{aligned}
\end{equation*}
Rearranging yields
    $\bbE[Q_\ell - \calQ] \approx \bbE[Q_\ell - Q_{\ell-1}](1-s^{\alpha_M})^{-1}$.
To bound the squared bias by $\hat{b}_w \varepsilon^2$ we consider 
\begin{equation*}
    |\bbE[Q_\ell - \calQ]| \approx \big |\bbE[Q_\ell - Q_{\ell-1}]\big|\,|1-s^{\alpha_M}|^{-1} \leq \sqrt{\hat{b}_w} \, \varepsilon,
\end{equation*}
which yields the numerically computable condition
\begin{equation}
 |B^L| := \bigg|\frac{1}{M_L}\sum_{k=1}^{M_L} Q_L^{(k)} - Q_{L-1}^{(k)}\bigg|\,\leq \sqrt{\hat{b}_w} \,\varepsilon \,|1-s^{\alpha_M}|.
 \label{BeBa_eq:MLMC_bias_condition}
\end{equation}
For more robustness in this bias condition, it is recommended in \cite{BeBa_Giles2015_MLMC}, to expand it to the last three computed terms, corrected with the expected bias decrease rate $\alpha_M$:
\begin{equation}
    \max \{s^{-2 \alpha_M} |B^{L-2}|; s^{- \alpha_M} |B^{L-1}| ; |B^L| \} \leq \sqrt{\hat{b}_w} \,\varepsilon \,|1-s^{\alpha_M}|.
    \label{BeBa_eq:bias-estimate-MLMC}
\end{equation}
Condition $\eqref{BeBa_eq:bias-estimate-MLMC}$ is used as a stopping criterion in the MLMC algorithm \ref{BeBa_alg:MLMC}, which determines the maximal level of the simulation on the fly.
\begin{algorithm}
\caption{MLMC}
\label{BeBa_alg:MLMC}
\begin{algorithmic}
\REQUIRE $\varepsilon > 0$, $M_{\ell,ini} \in \N$ for $\ell = 1,2,...$ initial number of samples on each level, $L_{max} \in \N$, $s>1$, $\alpha_M>0$, $\hat{b}_w \in (0,1)$.
\FOR{$L = 1:3$}
    \STATE $M_{L,opt}\gets M_{L,ini}$
    \STATE Evaluate $M_{L,opt}$ samples of $Q_L$ and $Q_{L-1}$
    \STATE Compute $B^L = \frac{1}{M_{L,opt}}\sum_{k=1}^{M_{L,opt}} Q_L^{(k)} - Q_{L-1}^{(k)}$
\ENDFOR
\\
 \WHILE{$\max \{s^{-2 \alpha_M} |B^{L-2}|; s^{- \alpha_M} |B^{L-1}| ; |B^L| \} > \sqrt{\hat{b}_w} \,\varepsilon \,|1-s^{\alpha_M}|$}
 \STATE $L \gets L+1$
 \STATE $M_{L,opt}\gets M_{L,ini}$
 \STATE Evaluate $M_{L,opt}$ samples of $Q_L$ and $Q_{L-1}$
    \FOR{$\ell = 1:L$}
        \STATE Estimate $\bbV_\ell \approx \bbV[Q_\ell - Q_{\ell-1}]$ with $M_{\ell,opt}$ samples
    \ENDFOR
    \FOR{$\ell = 1:L$}
    \STATE Determine $M_{\ell,tmp}$ by \eqref{BeBa_eq:opt-samples-MLMC}
        \IF{$M_{\ell,opt} < M_{\ell,tmp}$}
            \STATE Evaluate $M_{\ell,tmp} - M_{\ell,opt}$ more samples of $Q_\ell$ and $Q_{\ell-1}$
        \ENDIF
        \STATE $M_{\ell,opt} \gets M_{\ell,tmp}$
    \ENDFOR
    \STATE Compute $B^L = \frac{1}{M_{L,opt}}\sum_{k=1}^{M_{L,opt}} Q_L^{(k)} - Q_{L-1}^{(k)}$
 \ENDWHILE
 \STATE $\widehat{Q}_{0,L}^\text{MLMC} = \sum_{\ell=1}^L \frac{1}{M_{\ell,opt}} \sum_{k=1}^{M_{\ell,opt}} Q_\ell^{(k)} - Q_{\ell-1}^{(k)}$.
 \end{algorithmic}
\end{algorithm}
\subsection{Continuous Level Monte Carlo method}
\label{BeBa_subsec:CLMC}
The continuous level Monte Carlo (CLMC) method was introduced in \cite{BeBa_Detommaso2019_CLMC}. Just like MLMC, it is a method to estimate the mean value of a quantity of interest, but it allows for samplewise adaptive mesh hierarchies. This is realized by extending the integer level framework of MLMC to a continuous resolution (level) $\ell \in \R>0$  and a continuous family of approximations $(Q(\ell))_{\ell \geq 0}$ viewed as a stochastic process.
We formulate the method in its unbiased version, cf.~\cite[Corollary~$2.2$]{BeBa_Detommaso2019_CLMC}, since this fits the setting of our provided examples.
The unbiased CLMC estimator is defined by
\begin{equation}
    \widehat{Q}_{0,\infty}^\text{CLMC} := \frac{1}{M} \sum_{k=1}^M \int_0^{\infty} \frac{1}{\bbP(L_{r} \geq \ell)} \bigg(\frac{\textup{d}Q}{\textup{d}\ell}\bigg)^{(k)} (\ell) \ind_{[0,L_{r}^{(k)}]}(\ell) \textup{d}\ell,
\label{BeBa_eq:CLMC}
\end{equation}
with $M \in \N$ total samples and $L_{r} \sim \text{Exp}(r)$ with parameter $r > 0$, an exponentially distributed random variable independent of the stochastic process $(Q(\ell))_{\ell \geq 0}$. A realization $L_{r}^{(k)}$ corresponds to the maximal computed resolution of a sample $k$ for $1 \leq k \leq M$.
The CLMC method is an unbiased estimator for $\bbE[\calQ - Q_0]$, i.e., $\bbE[\widehat{Q}_{0,\infty}^\text{CLMC}] = \bbE[\calQ - Q_0]$ and thus, its MSE expansion \eqref{BeBa_eq:MSE} reduces to
 $\text{MSE}_{0,\infty}^\text{CLMC} := \bbV[\widehat{Q}_{0,\infty}^\text{CLMC}]$.
As for MLMC, a complexity theorem for CLMC has been proven in \cite[Theorem~$2.3$]{BeBa_Detommaso2019_CLMC}.
\begin{theorem}[CLMC - complexity theorem]
\label{BeBa_thm:CLMC-complexity}
Denote by $(Q(\ell))_{\ell \geq 0}$ a stochastic process defined on a probability space $(\widetilde{\Omega}, \widetilde{\calA}, \widetilde{\bbP})$ with $\bbE\big[\big|\frac{\textup{d}Q}{\textup{d}\ell}\big|\big] \in L^1(0,\infty)$, corresponding to a family of numerical approximations of $\calQ$ such that $Q(\ell) \to \calQ$ $\widetilde{\bbP}\textup{-almost surely}$ as $\ell \to \infty$. Suppose there exist positive constants $\alpha_C$, $\beta_C$, $\gamma_C$, $c_4$, $c_5$, $c_6$, with $\min\{\beta_C,2 \alpha_C\} > \gamma_C$ s.t. for any $\ell > 0$:
    \begin{subequations}
        \begin{equation}
                \bbE \bigg[\frac{\textup{d}Q}{\textup{d}\ell}\bigg] \leq c_4 e^{-\alpha_C \ell},
            \label{BeBa_eq:CLMC_meandecay}
        \end{equation}   
        \begin{equation}
            \bbV \bigg[\frac{\textup{d}Q}{\textup{d}\ell}\bigg] \leq c_5 e^{-\beta_C \ell}, 
            \label{BeBa_eq:CLMC_variancedecay}
        \end{equation}   
        \begin{equation}
            \calC[\ell] \leq c_6 e^{\gamma_C \ell},
            \label{BeBa_eq:CLMC_costincrease}
        \end{equation}   
    \end{subequations}
    where $\calC[\ell]$ is the cost to compute a sample of $Q(\ell)$ at resolution $\ell$.
    Let $\gamma_C < r < \min\{\beta_C,2 \alpha_C\}$, then for any $\varepsilon>0$, there exist $M \in \N$ and $C > 0$ such that 
    \begin{equation*}
        \text{MSE}_{0,\infty}^\text{CLMC} = \bbV[\widehat{Q}_{0,\infty}^\text{CLMC}] \leq \varepsilon^2 \quad\hbox{ and }\quad \calC[\widehat{Q}_{0,\infty}^\text{CLMC}] \leq C \varepsilon^{-2}.
    \end{equation*}
\end{theorem}
\begin{proof}
 The proof we give is for the special case of the unbiased estimator and includes more details on the variance estimate than in \cite{BeBa_Detommaso2019_CLMC}, to obtain a more accurate cost formula used in Section \ref{BeBa_subsec:optimal_exponential_distribution_parameter} to optimize the parameter $r$ of the exponential distribution. For better readability we give the proof under the further assumption that $r \notin \{ \frac{\beta_C}{2}, \beta_C, \frac{\alpha_C}{2}, \alpha_C\}$. These cases may be included by a further distinction of cases.
 As in \cite{BeBa_Detommaso2019_CLMC}, the variance is first split up by the law of the total variance into:
 \begin{equation*}
    \bbV[\widehat{Q}_{0,\infty}^\text{CLMC}] = \bbE\big[\bbV[\widehat{Q}_{0,\infty}^\text{CLMC}|L]\big] + \bbV\big[\bbE[\widehat{Q}_{0,\infty}^\text{CLMC}|L]\big].
 \end{equation*}
 These terms are then estimated separately. Since $L\sim \text{Exp}(r)$ is independent of the stochastic process $Q(\ell)$, the Cauchy--Schwarz inequality and Fubini's Theorem yield
 \begin{equation*}
  \begin{aligned}
    \bbE\big[\bbV[\widehat{Q}_{0,\infty}^\text{CLMC}|L]\big] = & \ \bbE\big[\text{Cov}[\widehat{Q}_{0,\infty}^\text{CLMC},\widehat{Q}_{0,\infty}^\text{CLMC}|L]\big] \\
    \leq & \ \frac{1}{M} \bbE \bigg[\int_{[0, L]^2} \frac{1}{\bbP(L\geq \ell)} \frac{1}{\bbP(L \geq \ell')} \bbV\bigg(\frac{\textup{d}Q}{\textup{d}\ell}\bigg)^\frac{1}{2} \bbV \bigg(\frac{\text{d}Q}{\text{d} \ell'} \bigg)^\frac{1}{2} \textup{d}\ell \textup{d}\ell' \bigg] \\ 
    =  & \ \frac{1}{M} \bbE \bigg[ \bigg(\int_0^{L} \frac{1}{\bbP(L\geq \ell)} \bbV\bigg(\frac{\textup{d}Q}{\textup{d}\ell}\bigg)^\frac{1}{2} \textup{d}\ell \bigg)^2 \bigg].
  \end{aligned}
 \end{equation*}
Next, the bound on the variance decay  \eqref{BeBa_eq:CLMC_variancedecay} is used to obtain
 \begin{equation*}
  \begin{aligned}
    \frac{1}{M} \bbE \bigg[ \bigg(\int_0^{L} & \frac{1}{\bbP(L\geq \ell)} \bbV\bigg(\frac{\textup{d}Q}{\textup{d}\ell}\bigg)^\frac{1}{2} \textup{d}\ell \bigg)^2 \bigg] \leq \ \frac{c_5}{M} \bbE \bigg[ \bigg(\int_0^{L} \frac{1}{\bbP(L\geq \ell)} e^{-\frac{\beta_C}{2} \ell} \textup{d}\ell \bigg)^2 \bigg] \\
    = & \ \frac{c_5}{M} \bbE \bigg[ \bigg(\int_0^{L}e^{(r -\frac{\beta_C}{2}) \ell} \textup{d}\ell \bigg)^2 \bigg] 
    =  \ \frac{1}{M} \frac{c_5}{(r - \frac{\beta_C}{2})^2} \bbE \big[(e^{(r - \frac{\beta_C}{2}) L} -1 )^2\big].
  \end{aligned}
 \end{equation*}    
 With the binomial formula this becomes
 \begin{equation*}
  \bbE \big[(e^{(r - \frac{\beta_C}{2}) L} -1 )^2\big] = \bbE \big[e^{(2 r - \beta_C) L}\big] - 2 \bbE \big[e^{(r - \frac{\beta_C}{2}) L}\big] + 1.
 \end{equation*}
 These terms are computed separately for $r < \beta_C$ using the exponential distributions density $f(\ell) = r e^{-r\ell}$:
 \begin{equation*}
 \begin{aligned}
  \bbE \big[e^{(2 r - \beta_C) L}\big] = & \ r\int_0^\infty e^{(r - \beta_C) \ell}\textup{d}\ell = \frac{r}{\beta_C - r} \\
  \quad\hbox{ and }\quad
  \bbE \big[e^{(r - \frac{\beta_C}{2}) L}\big] = & \  r\int_0^\infty e^{- \frac{\beta_C}{2} \ell}\textup{d}\ell = \frac{2 r}{\beta_C}.
  \end{aligned}
 \end{equation*}
 The second term is estimated next, using the fact that $\bbV[X - 1] = \bbV[X]$ for any random variable $X$. For $r <2\alpha_C$ we compute
 \begin{equation*}
 \begin{aligned}
  \bbV\big[\bbE[\widehat{Q}_{0,\infty}^\text{CLMC}|L]\big] = & \ \frac{1}{M} \bbV \bigg[ \int_0^{L} \frac{1}{\bbP(L\geq \ell)} \bbE\bigg(\frac{\textup{d}Q}{\textup{d}\ell}\bigg) \textup{d}\ell\bigg]  \leq \frac{c_4^2}{M} \bbV \bigg[ \int_0^{L} e^{(r - \alpha_C)\ell} \textup{d}\ell\bigg] \\
 = & \ \frac{c_4^2}{M(r-\alpha_C)^2}\Big( \bbE \big[e^{2(r - \alpha_C)L} \big] - \bbE\big[e^{(r - \alpha_C)L} \big]^2\Big).
 \end{aligned}
 \end{equation*}
 These terms are computed separately
 \begin{equation*}
 \begin{aligned}
\bbE \big[e^{2(r - \alpha_C)L} \big] = & \ r\int_0^\infty e^{(r - 2\alpha_C) \ell}\textup{d}\ell = \frac{r}{2\alpha_C - r} \\
\quad\hbox{ and }\quad
\bbE\big[e^{(r - \alpha_C)L} \big]^2 = & \ r\int_0^\infty e^{-\alpha_C \ell}\textup{d}\ell= \frac{r^2}{\alpha_C^2}.
 \end{aligned}
 \end{equation*}
Overall, using $\frac{r}{\beta_C -r} + 1 = \frac{\beta_C}{\beta_C - r}$ we obtain with $r<\beta_C, 2\alpha_C$:
\begin{equation*}
\begin{aligned}
    \bbV[\widehat{Q}_{0,\infty}^\text{CLMC}] \ & \leq  \frac{c_5}{M (r - \frac{\beta_C}{2})^2}\bigg(\frac{\beta_C}{\beta_C - r} - \frac{4 r}{\beta_C}\bigg) +  \frac{c_4^2}{M(r-\alpha_C)^2}\bigg(\frac{r}{2\alpha_C - r} - \frac{r^2}{\alpha_C^2}\bigg) \\
     \ & = \frac{1}{M} \bigg(\frac{4 c_5}{(\beta_C - r) \beta_C} +  \frac{c_4^2 r}{(2\alpha_C - r) \alpha_C^2} \bigg).
\end{aligned}
    \label{BeBa_eq:CLMC-variance}
\end{equation*}
For the algorithm cost, the unbiased estimator yields with $r > \gamma_C$
\begin{equation}
\begin{aligned}
 \calC[\widehat{Q}_{0,\infty}^\text{CLMC}] = & \ M \,\bbE\bigg[ \int_0^{L} \calC[\ell] \textup{d}\ell\bigg] \leq \frac{M c_3}{\gamma_C} \bbE \big[e^{\gamma_C L} - 1 \big] \\
 = & \ \frac{M c_6}{\gamma_C} \bigg(r \int_0^{\infty} e^{(\gamma_C - r)\ell} \textup{d}\ell - 1\bigg) =  \frac{M c_6}{r - \gamma_C}.
 \end{aligned}
 \label{BeBa_eq:CLMC-cost}
\end{equation}
 To obtain a variance and equivalently a MSE smaller than $\varepsilon^2$ we have to choose the sample size as
 \begin{equation}
  M \geq \bigg\lceil \frac{1}{\varepsilon^2} \bigg(\frac{4 c_5}{(\beta_C - r) \beta_C} +  \frac{c_4^2 r}{(2\alpha_C - r) \alpha_C^2} \bigg)\bigg\rceil.
  \label{BeBa_eq:CLMC-sample-size}
 \end{equation}
 Inserting this into the cost formula finishes the proof.
\end{proof}
\subsection{Optimal exponential distribution parameter}
\label{BeBa_subsec:optimal_exponential_distribution_parameter}
 The optimal choice of the parameter $r$ in the exponential distribution is crucial for an optimal time to error performance of the CLMC method. We demonstrate a way to determine an optimal parameter, which minimizes the total cost of CLMC while achieving a desired tolerance. Treating the sample size \eqref{BeBa_eq:CLMC-sample-size} as a continuous quantity and inserting it into the cost formula \eqref{BeBa_eq:CLMC-cost} yields
 \begin{equation}
 \begin{aligned}
  \ \calC[\widehat{Q}_{0,\infty}^\text{CLMC}] & \ \leq\frac{c_6}{\varepsilon^2 (r - \gamma_C)}\bigg(\frac{4 c_5}{(\beta_C - r) \beta_C} +  \frac{c_4^2 r}{(2\alpha_C - r) \alpha_C^2} \bigg)\\
  & \ =: \widehat{\calC}^\text{CLMC}(\varepsilon,c_4,c_5,c_6,\alpha_C,\beta_C,\gamma_C,r).
  \label{BeBa_eq:CLMC-cost-optimize}
  \end{aligned}
 \end{equation}
This cost formula depends on the given tolerance $\varepsilon$, the estimated parameters $c_4$, $c_5$, $c_6$, $\alpha_C$, $\beta_C$, $\gamma_C$ from Section \ref{BeBa_subsec:CLMC_parameter_estimates} and the exponential distribution parameter $r \in \bbR_{>0}$.
Using the provided lower and upper bound for $r$ from the complexity theorem \ref{BeBa_thm:CLMC-complexity} the cost is minimized by determining
 \begin{equation}
        \hat{r} = \argmin_{\gamma_C < r < \min\{\beta_C,2 \alpha_C\}} \widehat{\calC}^\text{CLMC}(\varepsilon,c_4,c_5,c_6,\alpha_C,\beta_C,\gamma_C,r).
    \label{BeBa_eq:CLMC_optimal_r}
 \end{equation}
\subsection{CLMC on the fly algorithm}
\label{BeBa_subsec:CLMC_algorithm}
The CLMC Algorithm \ref{BeBa_alg:CLMC} that we use in our numerical experiments uses an on the fly estimate of the variance $\bbV[\widehat{Q}_{0,\infty}^\text{CLMC}]$ as an estimator for the $\text{MSE}$ in each simulation run.
As in \cite{BeBa_Detommaso2019_CLMC}, to obtain a numerically computable estimator we define 
\begin{equation*}
 \bigg(\frac{\textup{d}Q}{\textup{d}\ell}\bigg)^{(k)} (\ell) := \frac{Q_j^{(k)} - Q_{j-1}^{(k)}}{\ell_{j}^{(k)} - \ell_{j-1}^{(k)}} \quad\hbox{ for }\quad \ell \in (\ell_{j-1} ,\ell_{j}),
\end{equation*}
through linear interpolation with samples $Q_j^{(k)}$ as approximations to the quantity of interest at levels $\ell_j$ for $j\geq 1$.
Using this definition together with $L_r \sim \text{Exp}(r)$, the discrete CLMC estimator becomes
\begin{equation}
    \widehat{Q}_{0,\infty}^\text{CLMC} = \frac{1}{M} \sum_{k=1}^M \sum_{j=1}^{J^{(k)}} w_j^{(k)} \big(Q_j^{(k)} - Q_{j-1}^{(k)}\big),
 \label{BeBa_eq:CLMC-discretized}
\end{equation}
with
\begin{equation}
    w_j^{(k)} := \frac{\exp(r \tilde{\ell}_j^{(k)}) - \exp(r \ell_{j-1}^{(k)})}{r(\ell_{j}^{(k)} - \ell_{j-1}^{(k)})},
 \label{BeBa_eq:CLMC-weights}
\end{equation}
and
\begin{equation}
    J^{(k)} := \min\{j \geq 1: \ell_j^{(k)} \geq L_r^{(k)} \}, \quad \tilde{\ell}_j^{(k)} := \min\{\ell_j^{(k)}, L_r^{(k)}\}.
    \label{BeBa_eq:CLMC-max-level-index}
\end{equation}
With the definition
 \begin{equation*}
    Y^{(k)} := \sum_{j=1}^{J^{(k)}} w_j^{(k)}\big(Q_j^{(k)} - Q_{j-1}^{(k)}\big),
 \end{equation*}
 the discretized CLMC estimator \eqref{BeBa_eq:CLMC-discretized} can be interpreted as a MC estimate for some random variable $Y$ and an unbiased estimator for its variance is given by
\begin{equation}
    \bbV[\widehat{Q}_{0,\infty}^\text{CLMC}] \approx \frac{1}{M-1}\bigg(\frac{1}{M}\sum_{k=1}^M \big(Y^{(k)}\big)^2 - \bigg(\frac{1}{M}\sum_{k=1}^M Y^{(k)}\bigg)^2\bigg) \quad\hbox{ for }\quad M \in \N. 
 \label{BeBa_eq:CLMC-variance-estimate}
\end{equation}
The samplewise continuous level of refinement for each sample $k$ is defined by
\begin{equation}
    \ell_j^{(k)} (\omega):= - \log \bigg(\frac{e_j^{(k)}}{e_0^{(k)}} \bigg) \quad\hbox{ for }\quad j=0,...,J^{(k)} \text{ and } k=1,...,M,
    \label{BeBa_eq:CLMC-level}
\end{equation}
naturally providing values $\ell_0^{(k)} = 0$ for all $k = 1,...,M$. The values $(e_j^{(k)})_{j=0}^{J^{(k)}}$ are computable a-posteriori error estimators, as e.g. provided by Equation \eqref{BeBa_eq:final_estimator}, where we define
%
 $ e_j^{(k)} := \bigg(\sum_{K \in \calK_j^{(k)}} \big(\eta_K^{(k)}\big)^2\bigg)^\frac{1}{2}  \quad\hbox{ for }\quad j=0,...,J^{(k)}$ and $k=1,...,M$.
%
Here $\eta_K^{(k)}$ is from Equation \eqref{BeBa_eq:estimator_primal} and
$\calK_j^{(k)}$ is a mesh at level $j$. The main reason for using the relative error in the level formula \eqref{BeBa_eq:CLMC-level} is due to the inaccessible sample-dependent constant $C>0$ in Equation \eqref{BeBa_eq:final_estimator}, see \cite{BeBa_Detommaso2019_CLMC}.
\begin{remark}
  Cutting the refinement of certain samples, that reach an incomputable linear system of dimension $\text{DOF}_{max}$, resp. maximal level $\bar{L}$, introduces only a negligible bias for practical values of $\varepsilon$. The probability, that $L^{(k)} > \bar{L}$ is estimated, cf.~\cite{BeBa_Detommaso2019_CLMC}, for the exponential distribution as
  \begin{equation}
    M \,\widetilde{\bbP}(L \geq \bar{L}) = M \exp(-r \bar{L}).
  \end{equation}
  If this rare event occurs, we approximate $Q^{(k)}(\ell)\approx Q^{(k)}(\bar{L})$ for $\ell \in [\bar{L}, L^{(k)}]$ as suggested in \cite{BeBa_Detommaso2019_CLMC}.
\end{remark}
\begin{algorithm}
\caption{CLMC}
\label{BeBa_alg:CLMC}
\begin{algorithmic}
\REQUIRE $\varepsilon > 0$, $M_{ini} \in \N$ number of initial samples, $\text{DOF}_{max}$ maximal degrees of freedom on the computing device, optimal $\hat{r}$ by \eqref{BeBa_eq:CLMC_optimal_r}.

\STATE $\widehat{Q}_{0,\infty}^\text{CLMC} \gets 0$
\STATE $k \gets 1$
\WHILE{$\bbV[\widehat{Q}_{0,\infty}^\text{CLMC}] > \varepsilon^2$}
 \STATE Draw $L_{\hat{r}}^{(k)} \sim \text{Exp}(\hat{r})$
 \STATE $j \gets 0$
 \STATE $\ell_{tmp} \gets 0$
 \WHILE{$\ell_{tmp} \leq L_{\hat{r}}^{(k)}$ and $\text{DOF}_\ell < \text{DOF}_{max}$}
        \STATE Evaluate a sample of $Q_j^{(k)}$
        \STATE Compute $\ell_j^{(k)}$ by \eqref{BeBa_eq:CLMC-level}
        \STATE $\ell_{tmp} \gets \ell_j^{(k)}$
        \STATE $j \gets j+1$
 \ENDWHILE
 \STATE Compute weights $(w_j^{(k)})_{j=1}^{J^{(k)}}$ by \eqref{BeBa_eq:CLMC-weights}
  \STATE $\widehat{Q}_{0,\infty}^\text{CLMC} \gets\widehat{Q}_{0,\infty}^\text{CLMC} + \sum_{j=1}^{J^{(k)}}w_j^{(k)} \big(Q_j^{(k)} - Q_{j-1}^{(k)}\big)$
  \IF{$k \leq M_{ini}$}
	\STATE Set variance $\bbV[\widehat{Q}_{0,\infty}^\text{CLMC}] = 1$
  \ELSE
	\STATE Estimate variance $\bbV[\widehat{Q}_{0,\infty}^\text{CLMC}]$ by \eqref{BeBa_eq:CLMC-variance-estimate} with $k$ samples
  \ENDIF
 \STATE $k \gets k+1$
 \ENDWHILE
 \STATE $M \gets k$
 \STATE $\widehat{Q}_{0,\infty}^\text{CLMC} \gets \frac{\widehat{Q}_{0,\infty}^\text{CLMC}}{M}$ 
 \end{algorithmic}
\end{algorithm}
\section{Quasi continuous level Monte Carlo method}
\label{BeBa_sec:quasi_continuous_level_monte_carlo_method}
This section examines the generation of random samples of the (maximal level) distribution $L_r$ in CLMC, see Equation \eqref{BeBa_eq:CLMC}. The distribution of $L_r$ is supposed to mimic the exponential decay in the assumptions on the convergence of the quantity of interest just like the decreasing sample sizes of MLMC over the levels. Deviations from this distribution will definitely affect the performance of the CLMC estimator in practice.
The standard choice for generating random numbers on a computing device are pseudo-random number generators. They are important in practice for their speed in number generation and their reproducibility.
For practical applications with reasonable sample numbers of $10^2-10^4$ a pseudo-random number generator might under- or oversample certain regions of the distribution. We recall, that the exponentially drawn random numbers $L_r^{(k)}$ for each sample enter in the practical CLMC estimator \eqref{BeBa_eq:CLMC-discretized} indirectly as the upper index $J^{(k)}$ in the sample sum and determine the maximal level of refinement of each sample. Hence, an under- or oversampling of certain regions in the exponential distribution leads to, e.g., not enough or too many samples computed up to a high resolution, distorting the CLMC simulation.
We aim at reducing this distortion in the practical CLMC estimator as good as possible and propose to sample the exponentially distributed random numbers in CLMC by using quasi-random, $[0,1)$-uniformly distributed, low-discrepancy numbers in combination with a transformation. We call this approach quasi continuous level Monte Carlo (QCLMC).
To generate a sample $L_r^{(k)}$, first generate a $[0,1)$-uniformly distributed quasi-random number and transform it to be an exponentially distributed quasi-random number with parameter $r$ 
through the inverse cumulative distribution function of the exponential distribution
    $F^{-1}(x,r) = \frac{-\text{ln}(1-x)}{r}$ for $\quad 0 \leq x < 1$.
Since the (maximal level) distribution in CLMC is one-dimensional regardless of the PDE dimension, there is no additional cost to generating the exponentially distributed quasi-random numbers.

\begin{remark}
    A complexity theorem for the QCLMC method based on the concept of $F$-discrepancy, which explicitly treats the quasi-random numbers as a deterministic number sequence, can be found in \cite{BeBa_Beschle2023_QCLMC}. In this work, we remain with its motivation and introduction in this section and demonstrate the numerical benefits in Sections \ref{BeBa_subsec:pseudo-random_quasi-random_numbers} and  \ref{BeBa_subsec:performance_of_mlmc_clmc_qclmc} showing its advantages over CLMC and MLMC in terms of time to error performance.
\end{remark}

\section{Random PDE model}
\label{BeBa_sec:random_pde_model}
In order to demonstrate and compare the performances of MLMC, CLMC and QCLMC we consider the quantity of interest $\calQ$ to be a functional of the solution of a random PDE. To this end, we consider a random elliptic jump-diffusion PDE, cf. \cite{BeBa_Babuska1970_DiscCoeffFEM, BeBa_Barth2018_Elliptic} as a simple mathematical model for subsurface flow through porous media.
Let $(\Omega,\calA,\bbP)$ be a complete probability space and $\calD \subset \bbR^d$, $d=1,2,3$ be a bounded and connected Lipschitz domain. The linear, random elliptic PDE with solution $u: \Omega \times \calD \to \bbR$ is given by
\begin{equation}
    -\nabla \cdot (a(\omega,x) \nabla u(\omega,x)) = f(x) \quad \text{in} \; \Omega \times \calD,
\label{BeBa_eq:PDE}
\end{equation}
where $f: \calD \to \bbR$ is the source term and $a : \Omega \times \calD \to \bbR$ is the random coefficient, with spatial discontinuities of random position. 
The boundary $\partial \calD$  is assumed to be Lipschitz continuous and equipped with homogeneous Dirichlet boundary conditions
\begin{equation*}
    \begin{aligned}
    u(\omega,x) &= 0 \quad \text{on} \; \Omega \times \partial \calD.
    \end{aligned}
\end{equation*}
Let $\calT: \Omega \to \calB(\calD)$, $\omega \mapsto (\calT_1(\omega),\calT_2(\omega))$ be a random disjoint partition  of $\calD$, i.e., $\calD = \calT_1(\omega) \cup \calT_2(\omega)$ with $\calT_1(\omega) \cap \calT_2(\omega) = \emptyset$ for each $\omega \in \Omega$, and let $P \in \bbR_{>0}$ be a deterministic positive real number. Then, the coefficient is defined by
\begin{equation}
    a : \Omega \times \calD \to \bbR, \quad (\omega,x) \mapsto \ind_{\calT_1 (\omega)}(x) \,P + \ind_{\calT_2 (\omega)}(x) \,P^{-1}.
\label{BeBa_eq:coefficient}
\end{equation}
Two different examples for the coefficient that are considered in the simulations are:

\textbf{Example $1$: Box coefficient}
 The first example is generated by sampling coordinates $x,y \sim \calU([0.4,0.6]$ and an edge length $l \sim \calU([0.2,0.3])$ for each sample $\omega \in \Omega$, where $x,y$ are the centers of a square with corresponding edge length $l$. The outside of the box is $\calT_1(\omega)$ and the inside of the box is $\calT_2(\omega)$. This coefficient introduces a peak at the random box position with steep gradients in the solution, see Figure \ref{BeBa_fig:box_cross_coefficient} (left). Different values of $P$ result in different magnitudes of the peak.
 
\textbf{Example $2$: Cross coefficient}
 The second example is generated by sampling coordinates $x,y \sim \calU([0.4,0.6]$ for each sample $\omega \in \Omega$, where $x$ yields a vertical line through $(x,0)$ and $(x,1)$ and $y$ yields a horizontal line through $(0,y)$ and $(1,y)$. This splits the domain into four squares. Two diagonally opposing squares are joined together to be $\calT_1(\omega)$ and $\calT_2(\omega)$, respectively. This coefficient admits two peaks diagonally opposing each other, see Figure \ref{BeBa_fig:box_cross_coefficient} (right).
\begin{remark}
 Both these coefficients can be sampled directly without the need of an approximation, as done for, e.g. log-normal Gaussian random fields, cf.~\cite{BeBa_Barth2018_Elliptic}.
\end{remark}
\begin{figure}[tbhp]
\centering
 \includegraphics[width=0.51\textwidth]{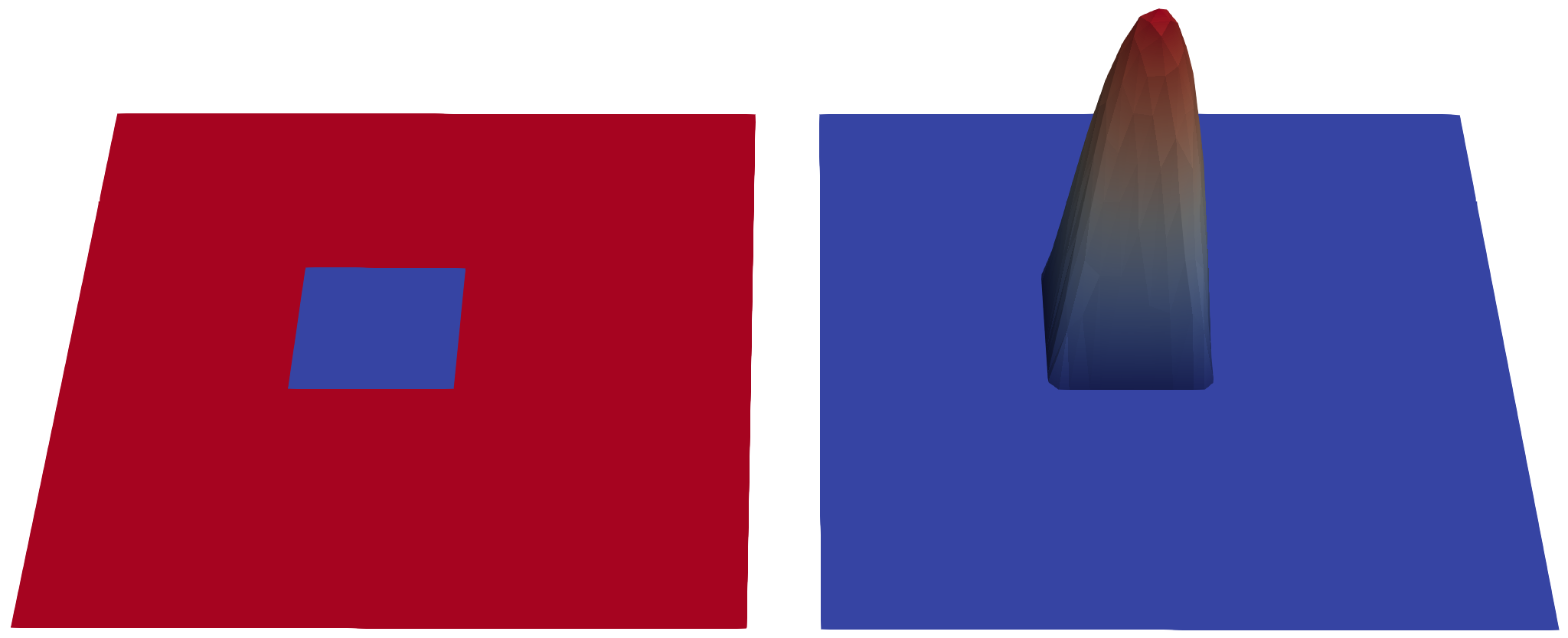}
 \includegraphics[width=0.48\textwidth]{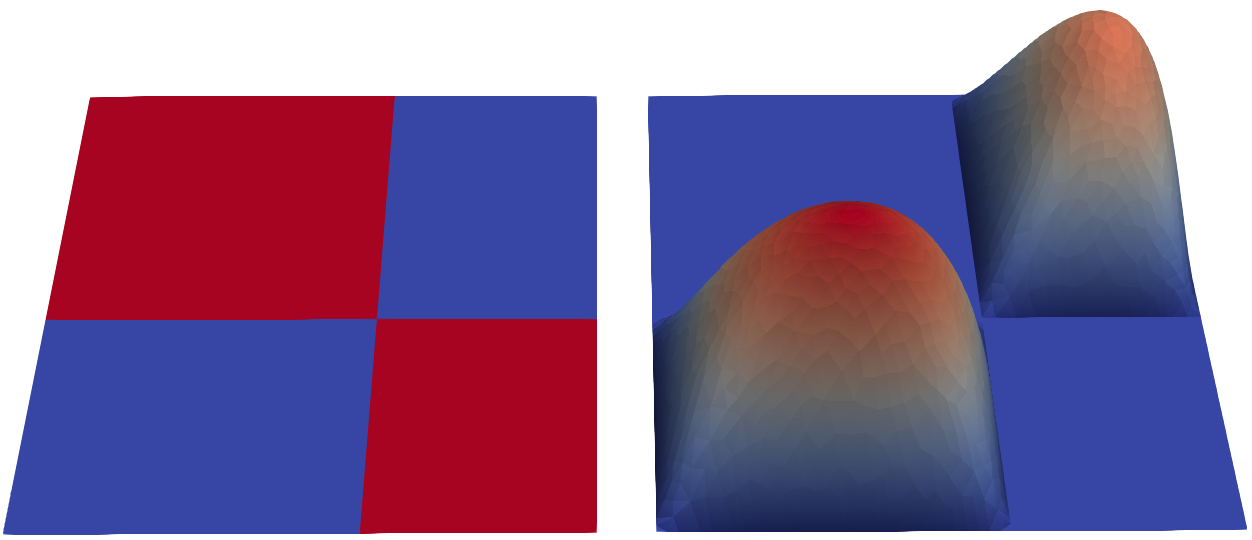}
 \caption{Visualization of single samples of the box coefficient with its PDE solution that has one peak (left) and the cross coefficient with its PDE solution that has two peaks (right).}
 \label{BeBa_fig:box_cross_coefficient}
\end{figure}
\begin{remark}
    The pathwise weak solutions in both examples exhibit regularity $u(\omega) \in H^{1+\kappa}(\calD)$, for $\mathbb{P}\textup{-almost all }\omega \in \Omega$, where $\kappa = \calO(P^{-1})$, see \cite[Lemma~$4.3$]{BeBa_Petzoldt2001_Regularity}. This implies deteriorated spatial convergence rates for numerical algorithms as the finite element method with uniform (un)structured meshes, cf.~\cite{BeBa_Babuska1970_DiscCoeffFEM}.
\end{remark}
\subsection{Pathwise weak formulation}
\label{BeBa_subsec:pathwise_weak_formulation}
To derive a pathwise weak formulation of the problem we recall the definition of standard Sobolev spaces $H^1(\calD)$, cf.~\cite[Section~$5.2.2$]{BeBa_Evans2010_PDE}, equipped with the norm
   $\|v\|_{H^1(\calD)}^2 := \int_{\calD} |v|^2 + \| \nabla v\|_2^2 \textup{d}x$ for $v \in H^1(\calD)$,
where $\| \cdot \|_2$ denotes the Euclidean norm on $\calD$. Let $H_0^1(\calD)$ be the space of $H^1(\calD)$-functions with vanishing trace on the boundary $\partial\calD$, $\omega \in \Omega$ be fixed and $f\in L^2(\calD)$. Multiplying the random PDE \eqref{BeBa_eq:PDE} by a test function $v \in H_0^1(\calD)$ and integrating over the spatial domain $\calD$ we obtain for $u(\omega) \in H_0^1(\calD)$,
		$\int_{\calD} a(\omega,x) \nabla u(\omega,x) \cdot \nabla v(x) \textup{d}x = \int_{\calD} f(x) v(x) \textup{d}x$,
with integration by parts.  
We define a suitable bilinear and linear form to ease the notation and write the weak formulation as
\begin{equation}
	B_{a(\omega)}(u(\omega),v) = F(v).
	\label{BeBa_eq:weak-form}
\end{equation}
Pathwise existence and uniqueness of a solution $u(\omega) \in H_0^1(\calD)$ follow by the Lax--Milgram lemma.
The bilinear form from Equation \eqref{BeBa_eq:weak-form} induces the energy norm via
\begin{equation}
   \|\cdot\|_E: H^1(\calD) \to \bbR, \quad v \mapsto \sqrt{B_{a}(v, v)} \quad\hbox{ for }\quad v \in H^1(\calD).
   \label{BeBa_eq:energynorm}
\end{equation}
\subsection{Finite element approximation}
\label{BeBa_subsec:finite_element_approximation}
For the numerical solution of the pathwise weak problem \eqref{BeBa_eq:weak-form} we consider the Finite Element method (FE), see, e.g., \cite[Chapters~$2$~and~$3$]{BeBa_Knabner2003_NumericalMethods}, \cite[Chapter~$8$]{BeBa_Hackbusch2017_EllipticEquations} and \cite{BeBa_Brenner2008_FEM}. In our approximation, we specifically allow for sample-dependent meshes with sample-dependent approximation spaces to generate the sample-adapted model hierarchies in the CLMC method, see Section \ref{BeBa_subsec:CLMC}. For fixed $\omega \in \Omega$ and some resolution $\ell \in \N$, let $\calK_\ell (\omega) := \bigcup_{i=1}^{m_\ell (\omega)} K_i (\omega)$ be a sample-dependent triangular mesh of the domain $\calD$ into elements $K_i$ for $i=1,...,m_\ell (\omega)$, with element diameter $h_{K_i}$. 
For $d=1$ this corresponds to an interval mesh and for $d=3$ to a tetrahedral mesh. Let us adopt the wording for $d=2$.
To find suitable FE approximations of $u (\omega)$ we use a standard Galerkin scheme with a finite dimensional subspace $V_\ell (\omega) \subset H_0^1(\calD)$ of resolution $\ell \in \N_0$ consisting of the classical piecewise linear basis functions. The discrete version of the weak formulation \eqref{BeBa_eq:weak-form} is to find $u_{\ell}(\omega) \in V_\ell (\omega)$ as a solution to
\begin{equation}
    B_{a(\omega)}(u_{\ell} (\omega), v_\ell (\omega)) = F (v_\ell (\omega)) \quad\hbox{ for }\quad v_\ell(\omega) \in V_\ell(\omega).
\label{BeBa_eq:weak-form-FEM}
\end{equation}
\section{A-posteriori error estimation}
\label{BeBa_sec:a_posteriori_error_estimation}
We aim to generate sample-dependent meshes to be used in the CLMC method. Therefore, we consider a-posteriori error estimation techniques. In FE analysis, a-posteriori error estimation is commonly used to guide mesh refinement procedures. The basic idea is to only refine elements of meshes that yield a high contribution to the error in the solution approximation. For an overview of different a-posteriori error estimation techniques see \cite{BeBa_Graetsch2005_APostError}. The motivation to use a-posteriori error estimators for our model problem stems from the fact, that the discontinuities in the coefficient imply easy to approximate flat areas as well as difficult to approximate steep gradients in the solution, as illustrated well by the examples in Figure \ref{BeBa_fig:box_cross_coefficient}. The flat areas allow for coarse meshes, whereas the steep gradients require high resolution to result in an accurate $H^1$-norm approximation. For $\omega \in \Omega$, the gradient of the numerical solution, $\nabla u_\ell (\omega)$, is an indicator for its steepness. It is included in a natural way in the $H^1$-norm of the numerical solution $\|u_\ell (\omega)\|_{H^1(\calD)}$ for all $\omega \in \Omega$. Thus, 
we define the $H^1$-norm as our quantity of interest
  $\calQ: H^1(\calD) \to \bbR$ with $ \calQ(u_\ell (\omega)) := \|u_\ell (\omega)\|_{H^1(\calD)}$ for $\omega \in \Omega$,
and state a samplewise a-posteriori estimator for its approximation error
\begin{equation}
    \big|\calQ(u(\omega)) - \calQ(u_{\ell}(\omega))\big| \;\big(= \big|\|u (\omega)\|_{H^1(\calD)} - \|u_\ell (\omega)\|_{H^1(\calD)}\big|\;\big)  \;\;\hbox{ for }\;\; \omega \in \Omega.
    \label{BeBa_eq:qoi_error}
\end{equation}
We achieve this by 
applying the inverse triangle inequality and the norm equivalence between the energy norm and the $H^1$-norm to an  a-posteriori error estimate in the energy norm for discontinuous coefficients, see \cite{BeBa_Bernardi2000_APostDisc, BeBa_Petzoldt2002_DiscCoeff}.
We fix $\omega \in \Omega$ and omit the dependence of quantities on $\omega \in \Omega$ in 
this section for a better readability. 
First, we compute
\begin{equation}
  \begin{aligned}
    \big|\|u\|_{H^1(\calD)} - \|u_\ell\|_{H^1(\calD)}\big| \leq \|u - u_\ell\|_{H^1(\calD)} \leq C_0 \|u - u_\ell\|_{E},
  \end{aligned}
  \label{BeBa_eq:error_representation}
\end{equation}
where $C_0>0$ is the norm equivalence constant independent of $u, u_\ell$. 

Next, we state the a-posteriori error estimator for the energy norm error for the 
problem \eqref{BeBa_eq:weak-form} in terms of numerically computable quantities. First, we give two small preliminary definitions.
Let $\calK$ be a mesh as described in Section \ref{BeBa_subsec:finite_element_approximation} and $v$ a function, that is piecewise constant on each element $K \in \calK$. The definition of the jump of $v$ across the edge $\gamma=K \cap K'$ of two adjacent elements $K$ and $K'$ is given by
 $[v]_\gamma = \lim_{t\to 0+} v(x + t \vv{n}_{\gamma}|_K) - \lim_{t\to 0+} v(x - t \vv{n}_{\gamma}|_K) 
            =  \lim_{t\to 0+} v(x + t \vv{n}_{\gamma}|_K) - \lim_{t\to 0+} v(x + t \vv{n}_{\gamma}|_{K'})$,
where $\vv{n}_{\gamma}|_K$ and $\vv{n}_{\gamma}|_{K'}$ are the respective unit normal vectors pointing in opposite directions. 
Further, we define the set containing the edges of an element $K \in \calK$, which are not on the boundary of the domain $\partial \calD$ by $\calE_{K}$.

The upcoming result, \cite[Theorem~$2.9$]{BeBa_Bernardi2000_APostDisc}, is an extension of standard residual based a-posteriori error estimates for the error in the energy norm, cf.~\cite{BeBa_Babuska1978_APostError,BeBa_Ainsworth1997_APostError}, to the linear elliptic PDE model with a piecewise constant and discontinuous coefficient.
%
\begin{theorem}[A-posteriori energy norm error estimator]
\label{BeBa_thm:energynorm-estimator-primal}
Let $\ell \in \N$ be arbitrary but fixed and let $\calK_\ell$ be a mesh as defined in Section \ref{BeBa_subsec:finite_element_approximation} that is aligned to the discontinuities in the coefficient \eqref{BeBa_eq:coefficient}. Let $u_{\ell} \in V_\ell$ be the computed piecewise linear finite element approximation of \eqref{BeBa_eq:weak-form-FEM}, $u \in H_0^1(\calD)$ be the unique weak solution of \eqref{BeBa_eq:weak-form} and let $f_\ell$ be a finite element approximation to $f$. Denote by $a_K$ the constant value of $a$ on element $K$ and set $a_\gamma = \max\{a_K,a_{K'}\}$ for each edge $\gamma = K \cap K'$ between two adjacent elements $K,K' \in \calK_\ell$. Further, denote by $h_K>0$ the diameter of element $K$ and by $h_\gamma>0$ the length of edge $\gamma$ and set
\begin{equation}
    \begin{aligned}
        \eta_K^2 :=  h_K^2 \, a_K^{-1} \|f_\ell\|_{L^2(K)}^2 + \frac{1}{2}\sum_{\gamma  \in \calE_K} h_\gamma a_\gamma^{-1}\|[\vv{n}_\gamma|_K a \nabla u_{\ell}]\|_{L^2(\gamma)}^2.
    \end{aligned}  
    \label{BeBa_eq:estimator_primal}
\end{equation}
There exist positive constants $C_1$ and $C_2$, depending only on $\calK_\ell$ such that
\begin{equation}
    \begin{aligned}
        \|u - u_{\ell}\|_E \leq C_1 \bigg\{ \sum_{K \in \calK_\ell}  \eta_K^2 + h_K^2 \, a_K^{-1}\|f - f_\ell \|_{L^2(K)}^2 \bigg\}^\frac{1}{2}, 
    \end{aligned}
    \label{BeBa_eq:upper_bound_1}
\end{equation}
and
\begin{equation}
    \eta_K \leq C_2 \bigg\{\|u - u_{\ell}\|_{E,w_K}^2 + \sum_{K' \in w_K} h_{K'}^2 a_{K'}^{-1} \|f - f_\ell\|_{L^2(K')}^2 \bigg\}^\frac{1}{2},
     \label{BeBa_eq:lower_bound_1}
\end{equation}
where $w_K$ denotes the union of all elements sharing an edge with $K$.
\end{theorem}
\begin{remark}
\label{BeBa_rmk:primal_estimates}
The novel idea in the proof of Theorem \ref{BeBa_thm:energynorm-estimator-primal} is a quasi-interpolation operator \cite{BeBa_Clement1975_Interpolation}, which allows for estimates on the interpolation error and multiplicative constants $C_1,C_2$ independent of the ratio $\max_{x\in\calD}{a(x)}/\min_{x\in\calD}{a(x)} = P^2$, cf~\cite[Lemma~$2.8$]{BeBa_Bernardi2000_APostDisc}. It requires a monotonicity assumption on the coefficient values with respect to the subdomains created by the discontinuities, see \cite[Hypothesis~$2.7$]{BeBa_Bernardi2000_APostDisc}. If this monotonicity is violated, the constants $C_1$ and $C_2$ will depend on the ratio $P^2$, cf.~ \cite[Remark~$2.10$]{BeBa_Bernardi2000_APostDisc}, but the estimates \eqref{BeBa_eq:upper_bound_1} and \eqref{BeBa_eq:lower_bound_1} still hold.
Further, Theorem \ref{BeBa_thm:energynorm-estimator-primal} assumes, that the mesh is aligned with the discontinuities. This assumption is not feasible in the cases when the discontinuities are curves, or when there is no a-priori knowledge of the discontinuity positions.
In both these cases we use an approximation $a_K$ to $a$ in the formula for $\eta_K$ given in Equation \eqref{BeBa_eq:estimator_primal} that is piecewise constant on each element $K \in \calK_\ell$. 
\end{remark}

Considering an exact approximation of $f$ by $f_\ell$, we obtain the following computable a-posteriori error estimate for the approximation of the $H^1$-norm as our quantity of interest
\begin{equation}
 \big|\|u\|_{H^1(\calD)} - \|u_\ell\|_{H^1(\calD)}\big| \leq C_0 C_1 \left( \sum_{K \in \calK_\ell} \eta_K^2 \right)^\frac{1}{2},
 \label{BeBa_eq:final_estimator}
\end{equation}
by combining the inverse triangle inequality and the norm equivalence estimate from \eqref{BeBa_eq:error_representation} with Theorem \ref{BeBa_thm:energynorm-estimator-primal}. This is used to drive the adaptive mesh marking and refinement procedure (Algorithm \ref{BeBa_alg:marking_algorithm}) using the D\"orfler marking strategy first introduced in \cite{BeBa_Doerfler1996_AdaptiveAlgorithm}, analysed in an abstract framework in \cite{BeBa_Carstensen2014_AxiomsOfAdaptivity} and applied in, e.g., \cite{BeBa_Gantner2022_StableImplementation}, yielding optimal convergence rates of the a-posteriori error estimator and the FE approximation in the energy, respectively $H^1$-norm. 

\begin{algorithm}
\caption{Marking and refinement algorithm}
\label{BeBa_alg:marking_algorithm}
\begin{algorithmic}

\REQUIRE Marking parameter $0<\vartheta<1$, number of refinements $L$, initial mesh $\calK_0$.
 \FOR{$\ell = 0,...,L-1$}
    \STATE Compute discrete solution $u_{\ell} \in V_\ell$.
    \STATE Compute refinement indicator $\eta_K$ for all $K \in \calK_\ell$ via Equation \eqref{BeBa_eq:estimator_primal}.
    \STATE Determine set $\widetilde{\calK}_{\ell}$ of elements $K \in \calK_\ell$ with smallest cardinality such that
    \begin{equation*}
     \vartheta \sum_{K \in \calK_\ell} \eta_K^2 \leq \sum_{K \in \widetilde{\calK}_{\ell}} \eta_K^2.
    \end{equation*}
    \STATE Refine all elements of the set $\widetilde{\calK}_{\ell}$ to obtain the refined mesh $\calK_{\ell+1}$.
 \ENDFOR
 
 \end{algorithmic}
\end{algorithm}

\section{Numerical experiments}
\label{BeBa_sec:numerical_experiments}
All computations in the upcoming sections are done in Python on an Intel(R) Core(TM) i$7$-$4770$ CPU running at $3.4$ GHz with $4$ cores and $2$ threads per core. All finite element computations are implemented via FEniCS \cite{BeBa_Fenics2015} and the linear systems are solved with its integrated optimized direct solver. The spatial domain is $\calD:=[0,1]^2$ and we set $f\equiv 1$ in \eqref{BeBa_eq:PDE} for simplicity. 
\subsection{Samplewise convergence on standard uniform vs. adaptively refined meshes} 
\label{BeBa_subsec:samplewise_convergence_on_standard_uniform_vs_adaptively_refined_meshes}
In this section we present the advantage of an adaptive refinement procedure by investigating the convergence properties of a single sample for each of the coefficients given by Examples $1$ and $2$ from Section \ref{BeBa_sec:random_pde_model} with $P=300$ on standard uniform and adaptively refined meshes. 
We start with the same initial mesh $\calK_0$ of $N_0=249$ vertices. The standard uniform meshes of finer resolution are attained through uniform refinement of the previous mesh with a scaling of $1.5$ in each space dimension resulting in $N_\ell \approx 249 \cdot 1.5^{2\cdot \ell}$ vertices at resolution $\ell \in \N$. The adaptively refined meshes of finer resolution are attained through adaptive refinement of the previous mesh according to Algorithm \ref{BeBa_alg:marking_algorithm} with $\vartheta = 0.5$. 
%
In Figure \ref{BeBa_fig:standard_adaptive_sample_convergence} we observe a substantial improvement of the convergence rate with the adaptively refined meshes (blue lines) by doubling the rate of convergence in comparison to the standard mesh convergence (red lines) for the approximation of the $H^1$-norm of the solution, i.e. $| \|u\|_{H^1(\calD)} - \|u_h\|_{H^1(\calD)}|$, which is the weak $H^1$-error.
The a-posteriori error estimator (light blue lines) converges with optimal rate $0.5$ as it is an upper bound to the strong $H^1$-error, i.e. $\|u - u_\ell\|_{H^1(\calD)}$.

Illustrations of the meshes are given in Figure \ref{BeBa_fig:standard_adaptive_meshes}, where we visually observe that the discontinuities are not resolved well by the coarse meshes. A high resolution is necessary at the discontinuities in order to approximate the PDE solution accurately. Realizing this with standard uniform meshes leads to large linear systems, which are expensive to solve.
In comparison, the adaptively refined meshes demonstrate the effectiveness of the 
a-posteriori error estimator by refining at the discontinuities of the coefficient and the peaks in the solution. Thus, in comparison to the standard uniform meshes the dimension of the linear systems to solve is reduced greatly while still retaining a high resolution of the meshes at the discontinuities and peaks.
\begin{figure}[tbhp]
  \includegraphics[width=0.49\textwidth]{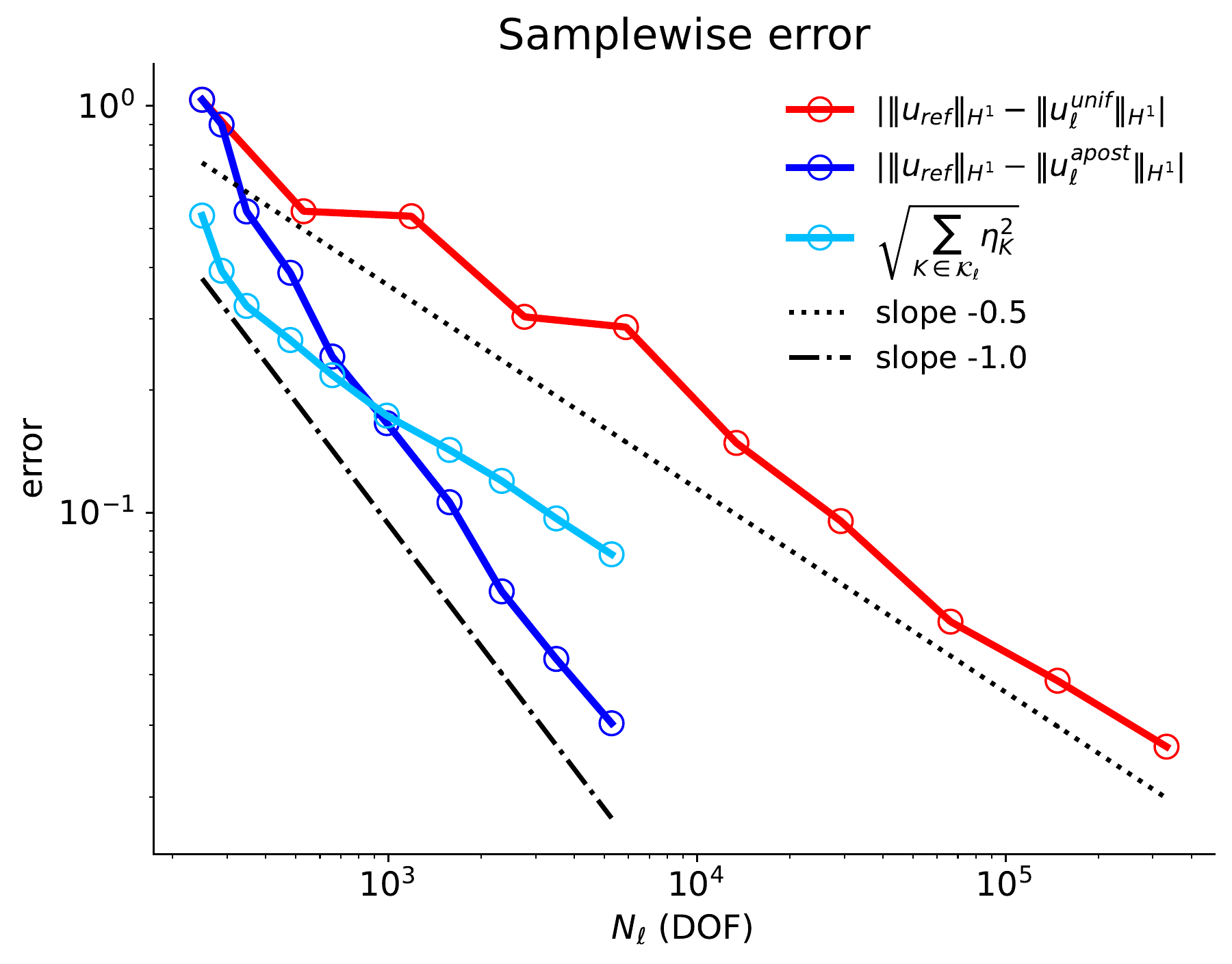}
 \includegraphics[width=0.49\textwidth]{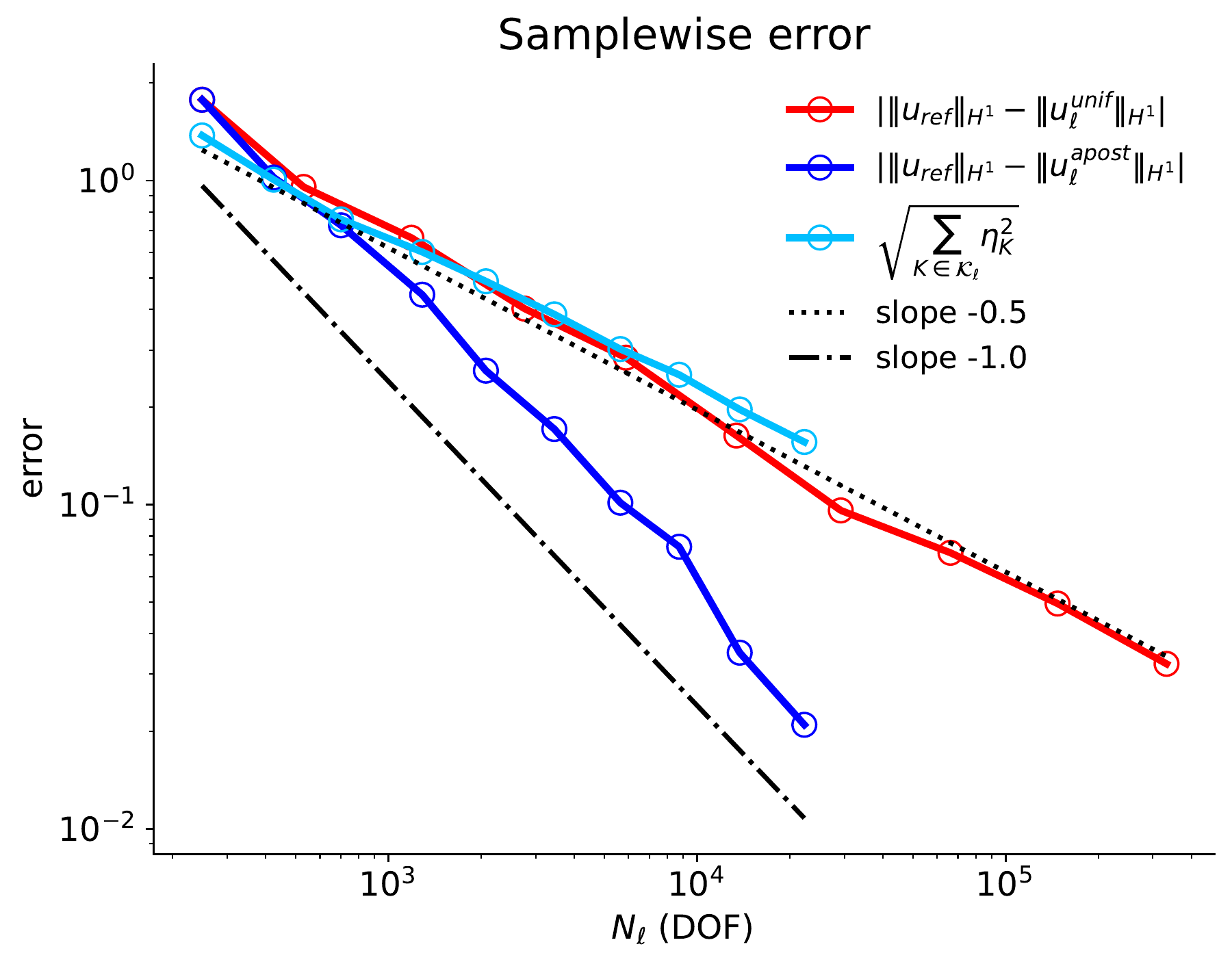}
  \caption{Convergence for box coefficient with $x=0.4,y=0.6,l=0.3$ and $P=300$ (left) and cross coefficient with $x=0.4,y=0.6$ and $P=300$ (right) on standard uniform vs. adaptively refined meshes meshes, refined with $\vartheta=0.5$. }
 \label{BeBa_fig:_sample_convergence}
 \label{BeBa_fig:standard_adaptive_sample_convergence}
\end{figure}
\begin{figure}[tbhp]
  \includegraphics[width=0.99\textwidth]{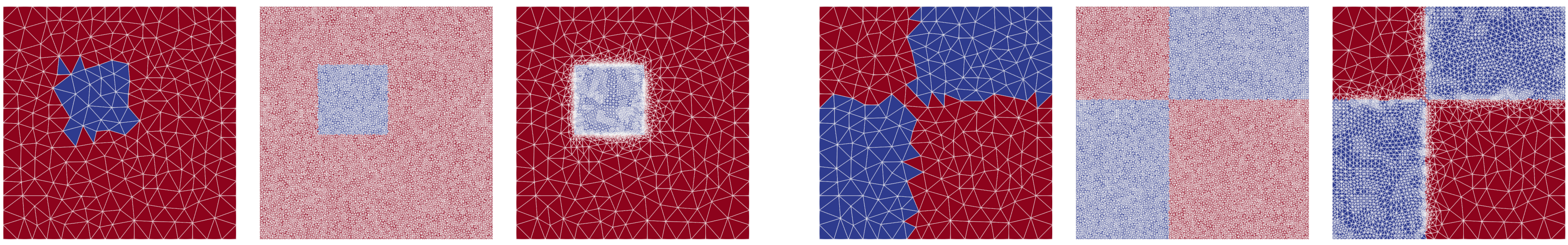}
 \caption{Left: Sample of box coefficient at resolutions $N_0=249$ and $N_5=13414$, $N_8=3498$. Right: Sample of cross coefficient at resolutions $N_0=249$ and $N_5=13414$, $N_5=3448$.}
 \label{BeBa_fig:standard_adaptive_meshes}
\end{figure}
\begin{remark}
\label{BeBa_rmk:adapted_reference}
 The reference solution for both examples has been computed on a mesh of resolution $N_{11} \approx 1850000$ and mesh elements aligned with the spatial discontinuities in the coefficient to achieve a high accuracy, cf.~\cite{BeBa_Barth2018_Elliptic}.
\end{remark}
\subsection{MLMC parameter estimates}
\label{BeBa_subsec:MLMC_parameter_estimates}
For the model problem considered in this work the constants $\alpha_M, \beta_M, \gamma_M, c_1,c_2,c_3$ from Theorem \ref{BeBa_thm:MLMC-complexity} are not available theoretically and need to be estimated numerically. Therefore, we apply the logarithm to base $s$ to Equations \eqref{BeBa_eq:MLMC_meandecay}, \eqref{BeBa_eq:MLMC_variancedecay} and \eqref{BeBa_eq:MLMC_costincrease} to obtain the linear relationships
\begin{equation}
\begin{aligned}
    \log_s(\bbE [Q_\ell - Q_{\ell-1}]) &\leq \tilde{c}_1 -\alpha_M \ell, \\ 
    \log_s(\bbV [Q_\ell - Q_{\ell-1}]) &\leq  \tilde{c}_2 -\beta_M \ell, \\
    \log_s(\calC[Q_\ell - Q_{\ell-1}]) &\leq  \tilde{c}_3 + \gamma_M \ell,
\end{aligned}
\end{equation}
where $\tilde{c}_i = \log_s(c_i)$ for $i=1,2,3$. On each level $\ell=1,...,L$, $M \in \N$ samples of the difference quantity  $Q_\ell^{(k)} - Q_{\ell-1}^{(k)}$ for $k=1,...,M$, are computed. The mean, variance and cost quantities are estimated by sample averages and are then used to compute the desired parameters and constants from the linear relationships by linear fitting.
\begin{figure}[tbhp]
 \includegraphics[width=0.49\textwidth]{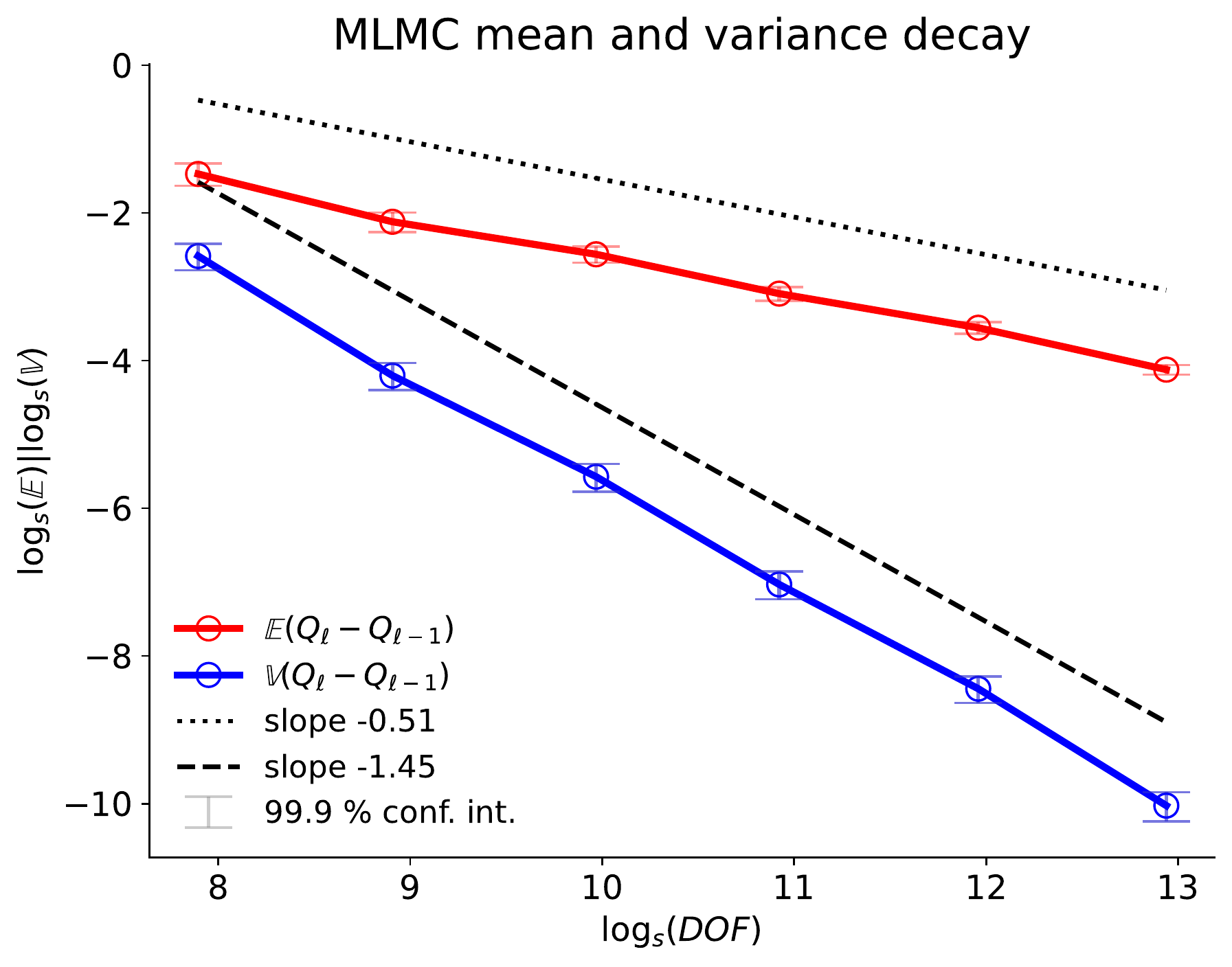}
 \includegraphics[width=0.49\textwidth]{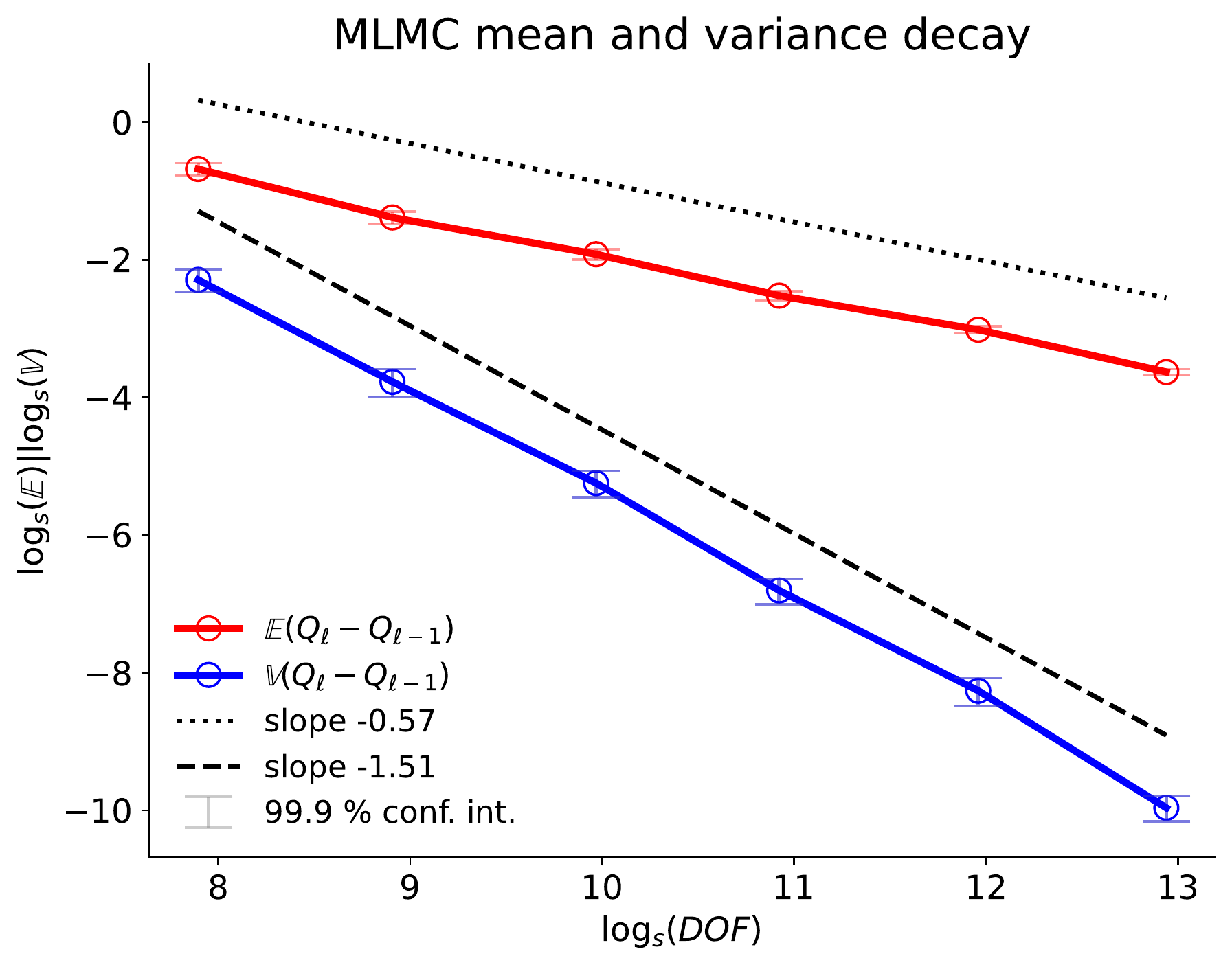}
 \caption{MLMC parameter estimates with $L=6$ and $M=1000$: box coefficient (left) and cross coefficient (right) for $P=300$. Simulations with $P=1000$ yielded about the same convergence rates, but with upward shifted curves.}
 \label{BeBa_fig:MLMC_parameter_estimates_box}
\end{figure}
\subsection{CLMC parameter estimates}
\label{BeBa_subsec:CLMC_parameter_estimates}
As for MLMC, the constants $\alpha_C, \beta_C, \gamma_C, c_4,c_5,c_6$ from Theorem \ref{BeBa_thm:CLMC-complexity} are not available theoretically and need to be estimated numerically. We apply the natural logarithm to Equations \eqref{BeBa_eq:CLMC_meandecay}, \eqref{BeBa_eq:CLMC_variancedecay} and \eqref{BeBa_eq:CLMC_costincrease} to obtain the linear relationships
\begin{equation}
\begin{aligned}
    \text{ln}(\bbE [\scalebox{0.8}{$\frac{\textup{d}Q}{\textup{d}\ell}$}]) &\leq \tilde{c}_4 -\alpha_C \ell, \\ 
    \text{ln}(\bbV [\scalebox{0.8}{$\frac{\textup{d}Q}{\textup{d}\ell}$}]) &\leq \tilde{c}_5 -\beta_C \ell, \\
    \text{ln}(\calC[\ell]) &\leq  \tilde{c}_6 + \gamma_C \ell,
\end{aligned}
\end{equation}
where $\tilde{c}_i = \text{ln}(c_i)$ for $i=4,5,6$. $M \in \N$ samples $k=1,...,M$ of the quotient $\bigg(\frac{\textup{d}Q}{\textup{d}\ell}\bigg)^{(k)} (\ell) := \frac{Q_j^{(k)} - Q_{j-1}^{(k)}}{\ell_{j}^{(k)} - \ell_{j-1}^{(k)}}$ are computed for $\ell \in (\ell_{j-1}^{(k)} ,\ell_{j}^{(k)})$ and $j=1,...,J$. Since each sample $k$ provides different values for the levels $\ell_j^{(k)}$ for $j\geq1$ we numerically interpolate the computed quotient onto a common level domain $[\max_k{\ell_1^{(k)}}, \min_k{\ell_J^{(k)}}]$. The mean, variance and cost quantities are estimated on this common level domain by sample averages and are then used to compute the desired parameters and constants from the linear relationships by linear fitting.
\begin{figure}[tbhp]
 \includegraphics[width=0.49\textwidth]{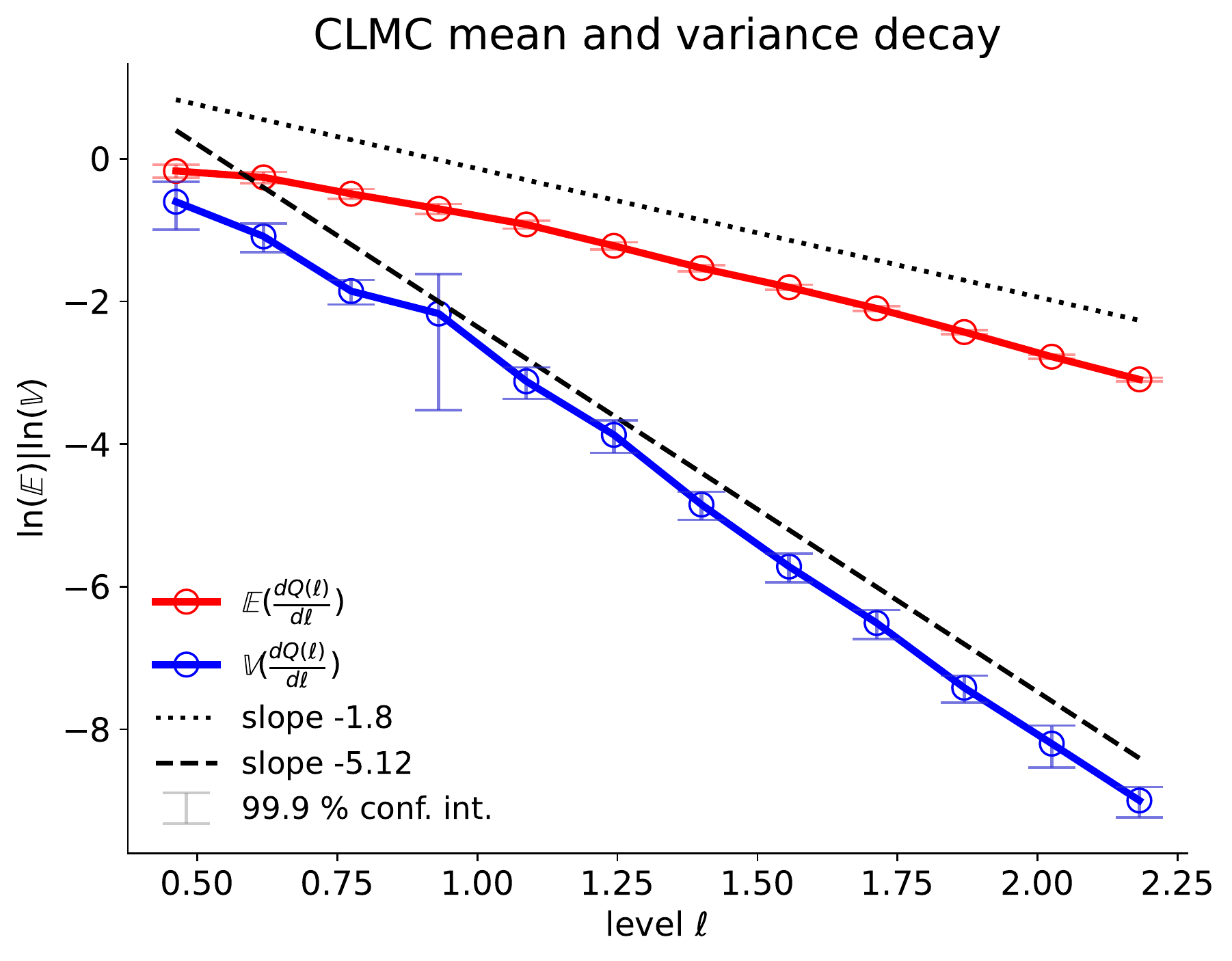}
\includegraphics[width=0.49\textwidth]{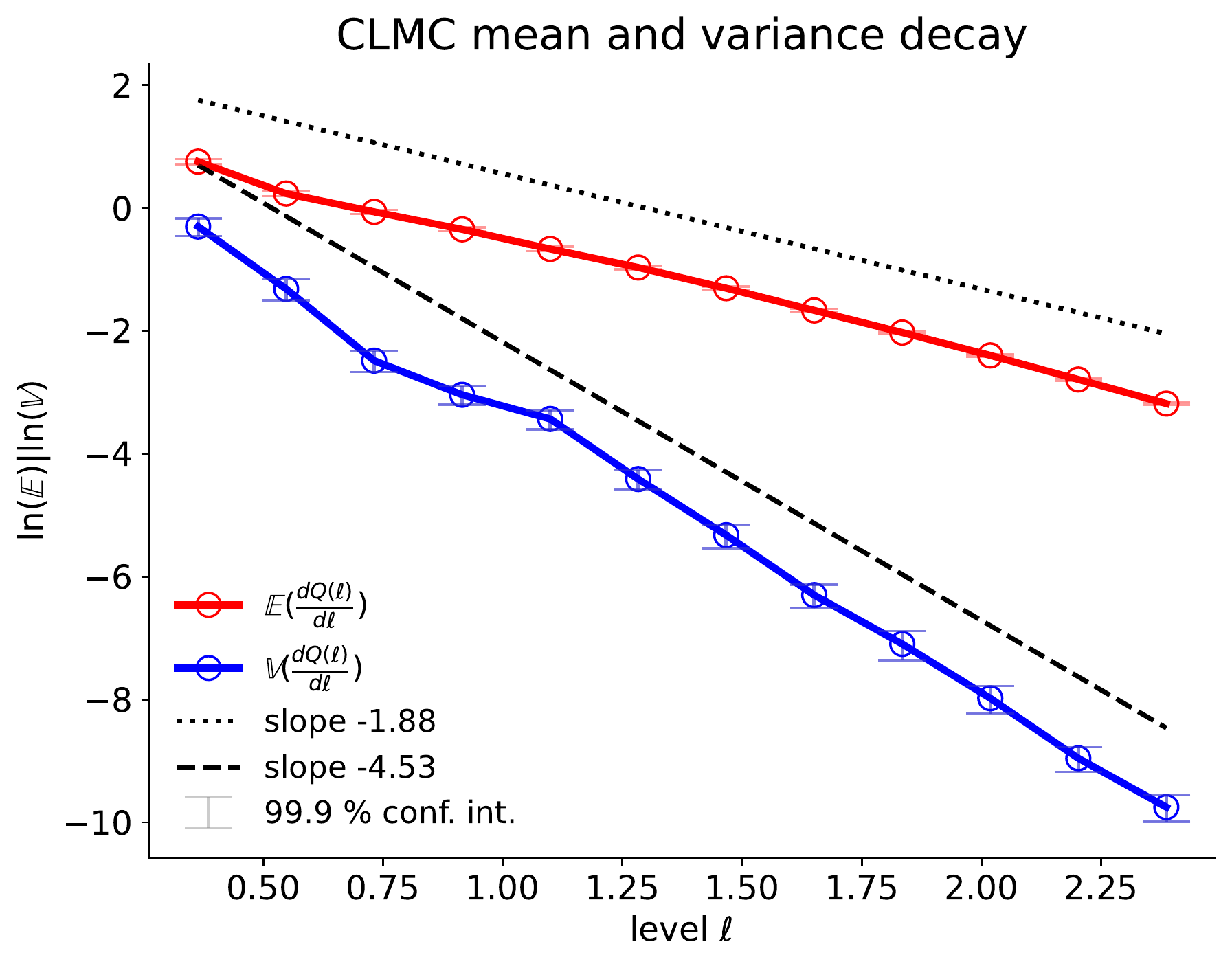}
 \caption{CLMC parameter estimates for $J=12$ and $M=1000$: box coefficient (left) and cross coefficient (right) for $P=300$. Simulations with $P=1000$ yielded about the same convergence rates, but with upward shifted curves.}
 \label{BeBa_fig:CLMC_parameter_estimates_box}
\end{figure}
\subsection{Pseudo-random vs. quasi-random numbers}
\label{BeBa_subsec:pseudo-random_quasi-random_numbers}
In this section we want to demonstrate the superiority of quasi-random numbers over pseudo-random numbers in approximating the one-dimensional exponential distribution with moderate sample sizes. 
A convergence experiment is visualized in Figure \ref{BeBa_fig:quasi_random_convergence}, where $\bbE_{exact}$ and $\bbV_{exact}$ are known values and $\bbE_{approx}$ and $\bbV_{approx}$ are computed using the indicated number of samples on the $x$-axis. The MSE on the $y$-axis is estimated over $20$ independent runs. We observe a much faster MSE decay in approximating the mean (right) and variance (left) of the exponential distribution with quasi-random Sobol numbers in comparison to pseudo-random numbers. For a visual comparison of the distributional properties of $2^{13}=8192$ samples of quasi-random Sobol numbers and pseudo-random numbers, see Figure \ref{BeBa_fig:quasi_random_comparison}. The quasi-random Sobol numbers yield visually very accurate approximations of the exponential distribution density function, whereas deviations are clearly visible for the pseudo-random numbers.
\begin{figure}[tbhp]
 \includegraphics[width=0.49\textwidth]{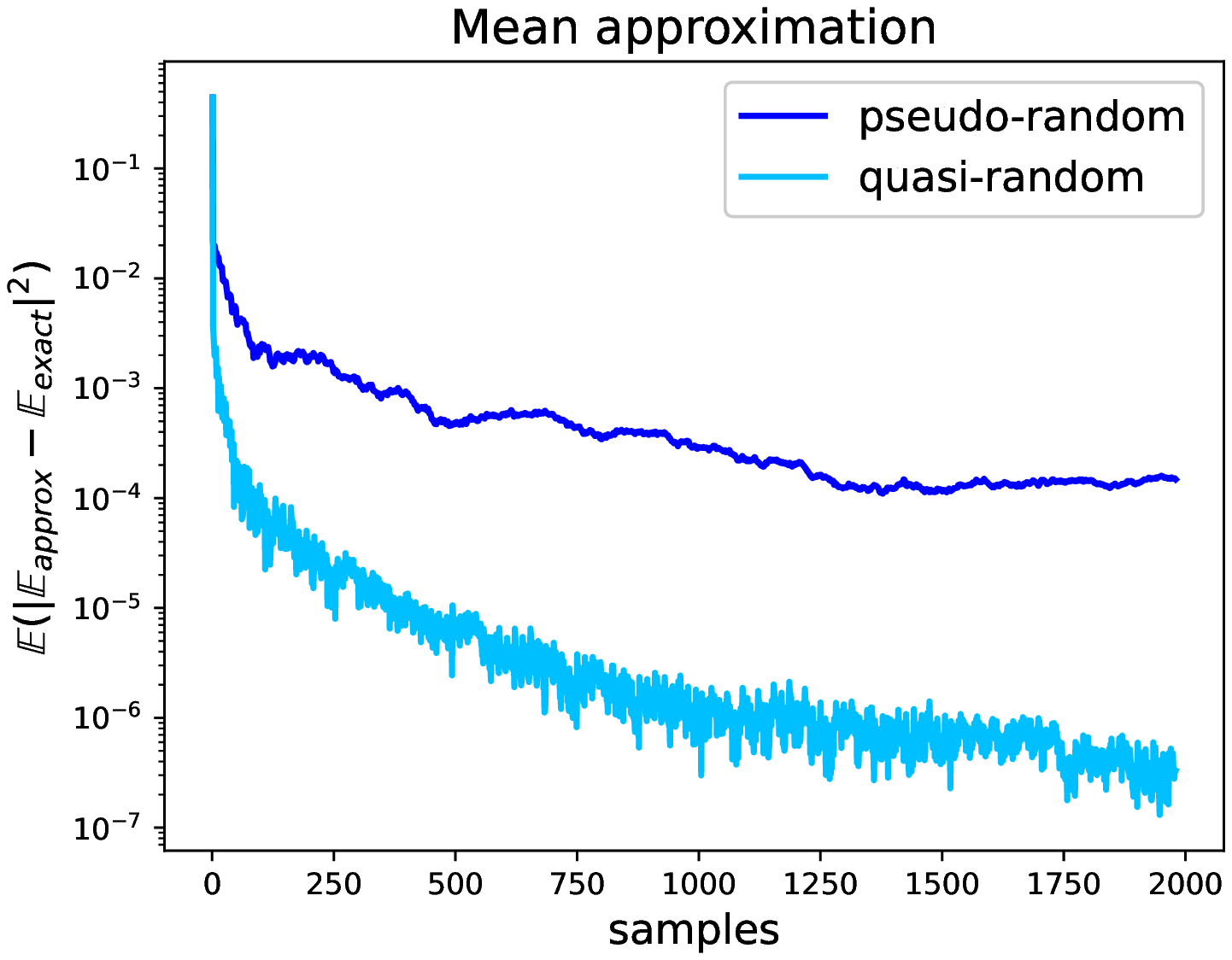}
 \includegraphics[width=0.49\textwidth]{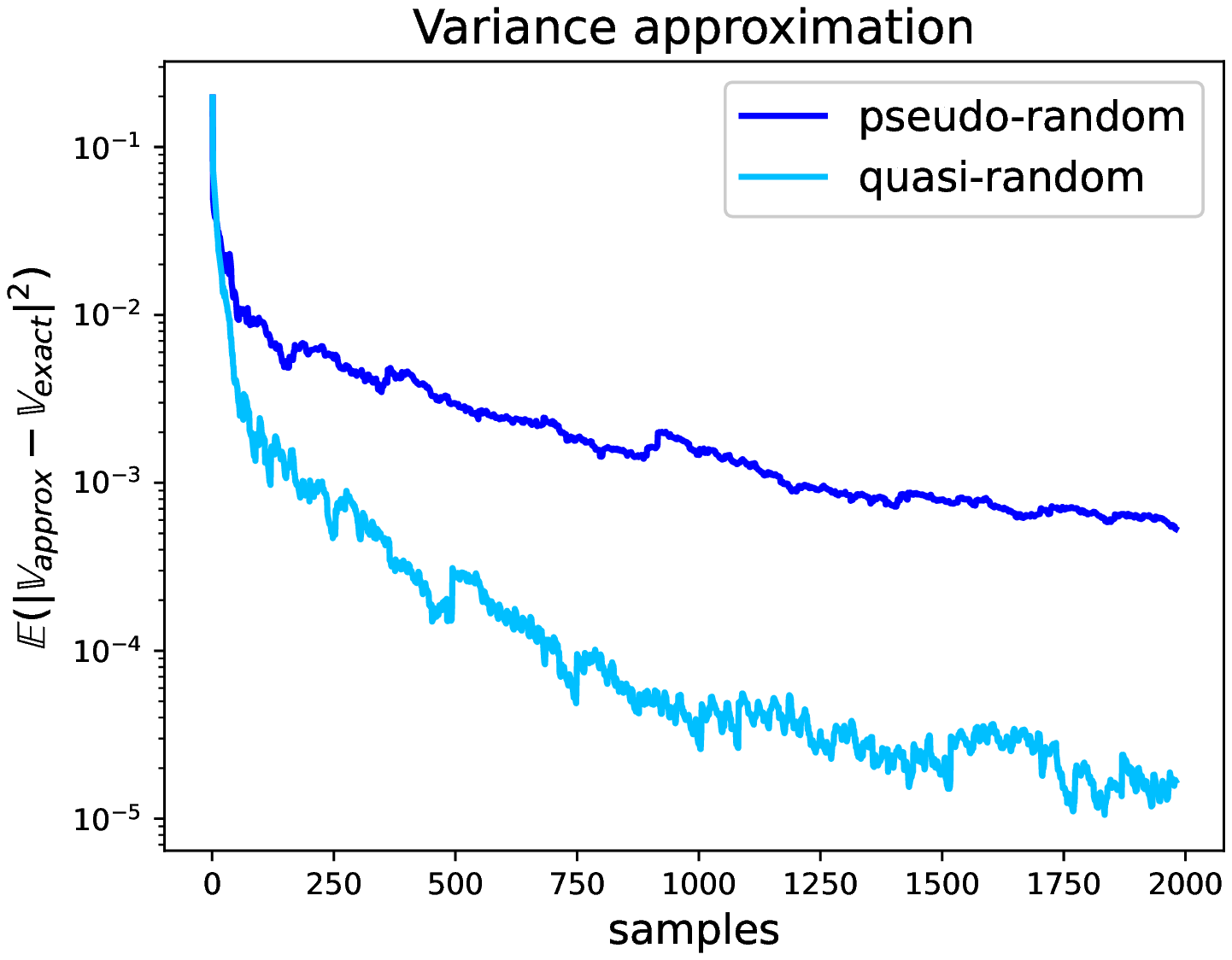}
 \caption{MSE for the approximation of the exact values $\bbE_{exact}(L_{1.5}) = \frac{2}{3}$ and $\bbV_{exact}(L_{1.5}) = \frac{4}{9}$ for $L_{1.5} \overset{d}{=} \text{Exp}(1.5)$.}
 \label{BeBa_fig:quasi_random_convergence}
\end{figure}
\begin{figure}[tbhp]
 \includegraphics[width=0.49\textwidth]{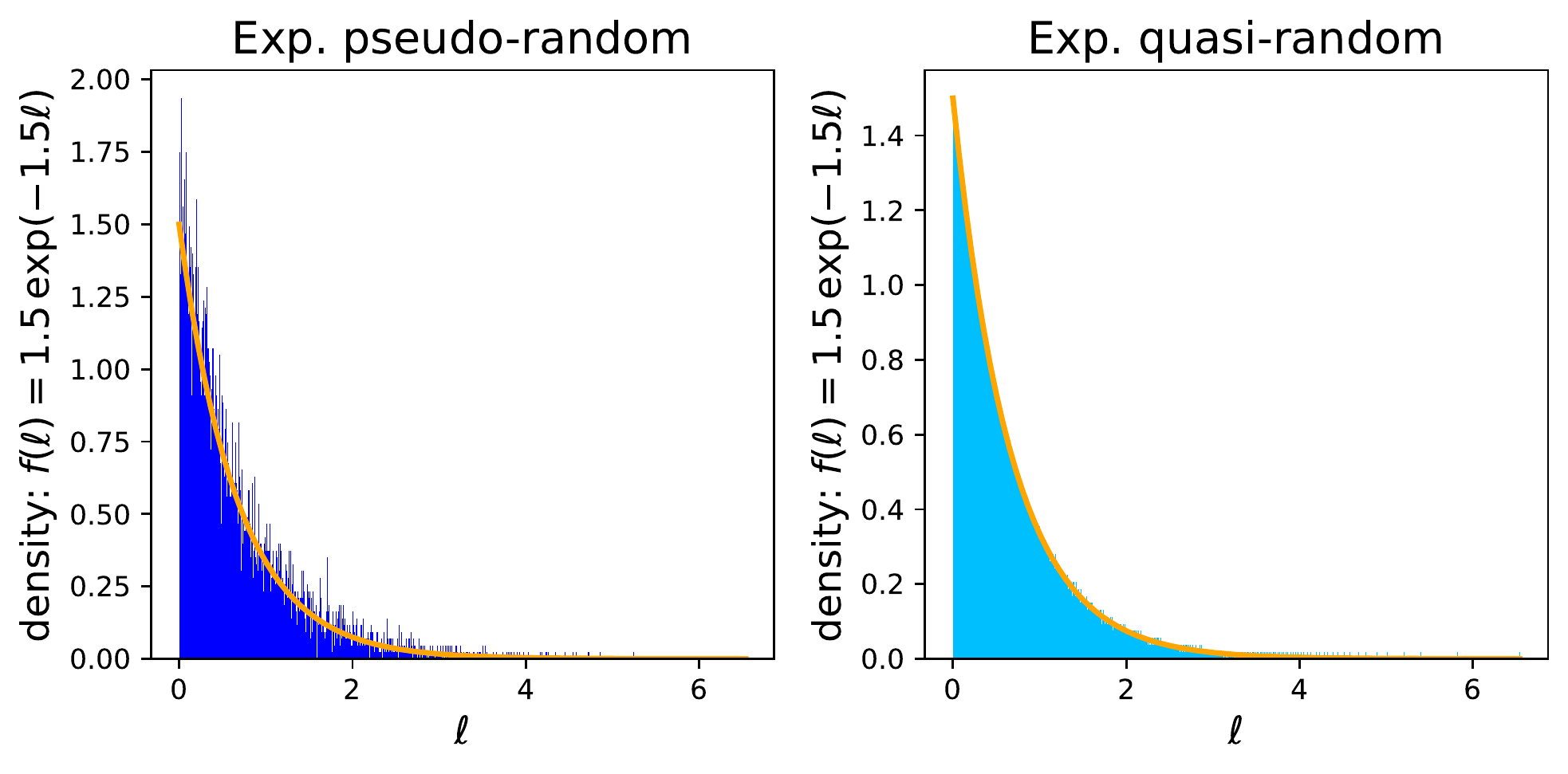}
 \includegraphics[width=0.49\textwidth]{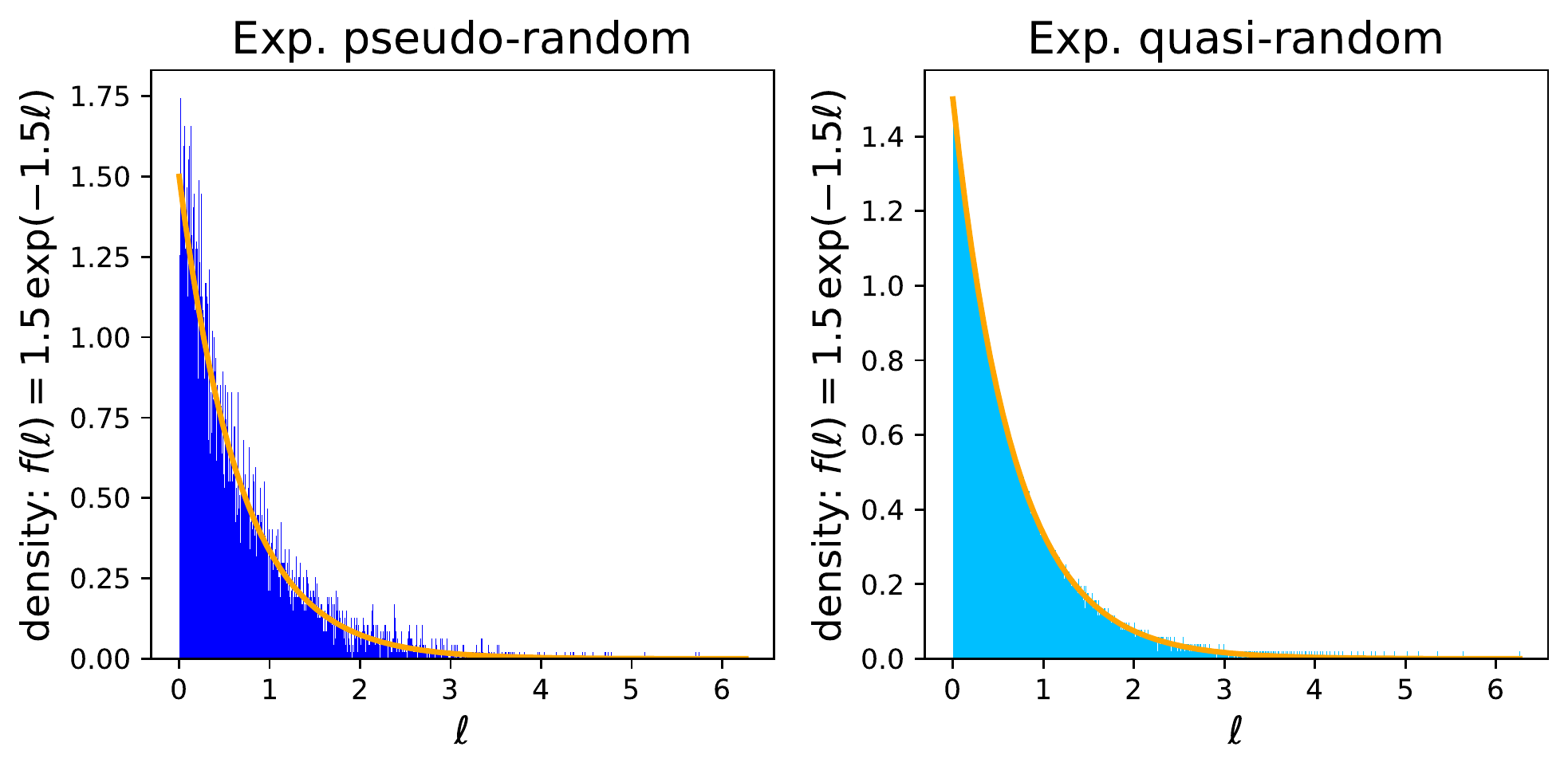}
 \caption{Comparison of quasi-random Sobol numbers and pseudo-random numbers for two different random seeds using $2^{13}=8192$ samples.}
 \label{BeBa_fig:quasi_random_comparison}
\end{figure}
\subsection{Performance of MLMC, CMLC and QCMLC}
\label{BeBa_subsec:performance_of_mlmc_clmc_qclmc}
In this subsection we conduct numerical experiments comparing the three introduced stochastic simulation methods multilevel Monte Carlo (MLMC), continuous level Monte Carlo (CLMC) and our variant quasi continuous level Monte Carlo (QCLMC) in their respective time to error performance. To obtain the standard uniform meshes for MLMC and the adaptive meshes for the CLMC methods, we proceed as outlined in Section \ref{BeBa_subsec:samplewise_convergence_on_standard_uniform_vs_adaptively_refined_meshes}.

We optimize the MLMC method in terms of time to error performance as proposed in Section \ref{BeBa_subsec:optimal_mse_weight} by computing an optimal weighting between the mean squared error contributions. The CLMC method is optimized in terms of time to error performance as proposed in Section \ref{BeBa_subsec:optimal_exponential_distribution_parameter} by computing an optimal exponential distribution parameter. The QCLMC method uses the same exponential distribution parameter as the CLMC method for a better comparability of the methods. 
We simulate the approximation of $\bbE[\calQ(u) - Q_0]$, where $\calQ(u):=\|u\|_{H_0^1(\calD)}$ with $u$ as the solution to Equation \eqref{BeBa_eq:weak-form} and $Q_0 := \calQ(u_0)$ an approximation involving the numerical solution $u_0$ to Equation \eqref{BeBa_eq:weak-form-FEM} on the coarsest mesh $\calK_0$.
\subsubsection{Reference solution}
\label{BeBa_subsec:reference_solution}
We do not have access to the exact value of $\bbE[\calQ(u) - Q_0]$ and approximate it to a very high accuracy.
We use MLMC with meshes aligned with the spatial discontinuities, cf.~\cite{BeBa_Barth2018_Elliptic} for details, to estimate $\bbE[\calQ]\approx \widehat{Q}_{L}^{\text{MLMC}_{adp}}$.
Then, we subtract an independent (non-adapted) MC estimator \eqref{BeBa_eq:montecarlo} for $\bbE[Q_0]\approx \widehat{Q}_0^\text{MC}$, which is unbiased.
We choose tolerances of $0.0025^2$ for the different terms in the MSE expansion and compute
\begin{equation*}
    \begin{aligned}
        \varepsilon_{ref}^2 = & \ \bbE\Big[\big(\bbE[\calQ - Q_0] - (\widehat{Q}_{L}^{\text{MLMC}_{adp}} - \widehat{Q}_0^\text{MC})\big)^2\Big] \\
        = & \ \bbV\Big[\widehat{Q}_{L}^{\text{MLMC}_{adp}}\Big] + \bbE\big[\calQ - Q_L\big]^2 + \bbV\Big[\widehat{Q}_0^\text{MC}\Big] \\
        \leq & \ 0.0025^2 + 0.0025^2 + 0.0025^2 = 3 \cdot 0.0025^2,
    \end{aligned}
\end{equation*}
for the total MSE of the reference solution.
\subsubsection{Method comparison}
\label{BeBa_subsec:method_comparison}
For each simulation method and each of the two coefficient examples from Section \ref{BeBa_sec:random_pde_model} we conducted $100$ independent runs, see Remark \ref{BeBa_rmk:independence}, on $7$ different MSE tolerances given by
\begin{equation}
 \varepsilon_i^2 = 0.04 \cdot 0.8^{2i}, \quad\hbox{ for }\quad i=0,...,6.
 \label{BeBa_eq:mse_tolerances}
\end{equation}
We remark that the MSE of the reference solution is more than $100$ times smaller.
Next, we estimate the real average MSE over $100$ independent runs by an MC estimate and compute confidence intervals via the central limit theorem for each estimation. 

In Figure \ref{BeBa_fig:convergence_and_solution_spread_box} (left) Example $1$ from Section \ref{BeBa_sec:random_pde_model} is simulated and we observe that all methods attain their expected time to MSE convergence rate of $1$. For $P=300$ MLMC performs a bit better than CLMC, but their $95\%$ confidence intervals are close to one another. QCLMC clearly outperforms both methods. For $P=1000$ MLMC and CLMC perform similarly, while QCLMC again outperforms both methods by a considerable margin. 
In Figure \ref{BeBa_fig:convergence_and_solution_spread_box} (right) we observe that the solutions to most of the simulation runs of QCLMC are much closer to the reference solution than the solutions to the runs of CLMC. This underlines the observed lower MSE in Figure \ref{BeBa_fig:convergence_and_solution_spread_box} of QCLMC.

In Figure \ref{BeBa_fig:convergence_and_solution_spread_cross} (left) Example $2$ from Section \ref{BeBa_sec:random_pde_model} is simulated and we observe that all methods attain their expected time to MSE convergence rate of $1$. For $P=300$ MLMC clearly outperforms CLMC, but is again outperformed by QCLMC. For $P=1000$ MLMC still outperforms CLMC, but less than before, while QCLMC again clearly outperforms both MLMC and CLMC. 
In Figure \ref{BeBa_fig:convergence_and_solution_spread_cross} (right) we observe as before that the QCLMC simulation runs yield consistently better results than the CLMC runs.
\begin{figure}[tbhp]
 \includegraphics[width=0.49\textwidth]{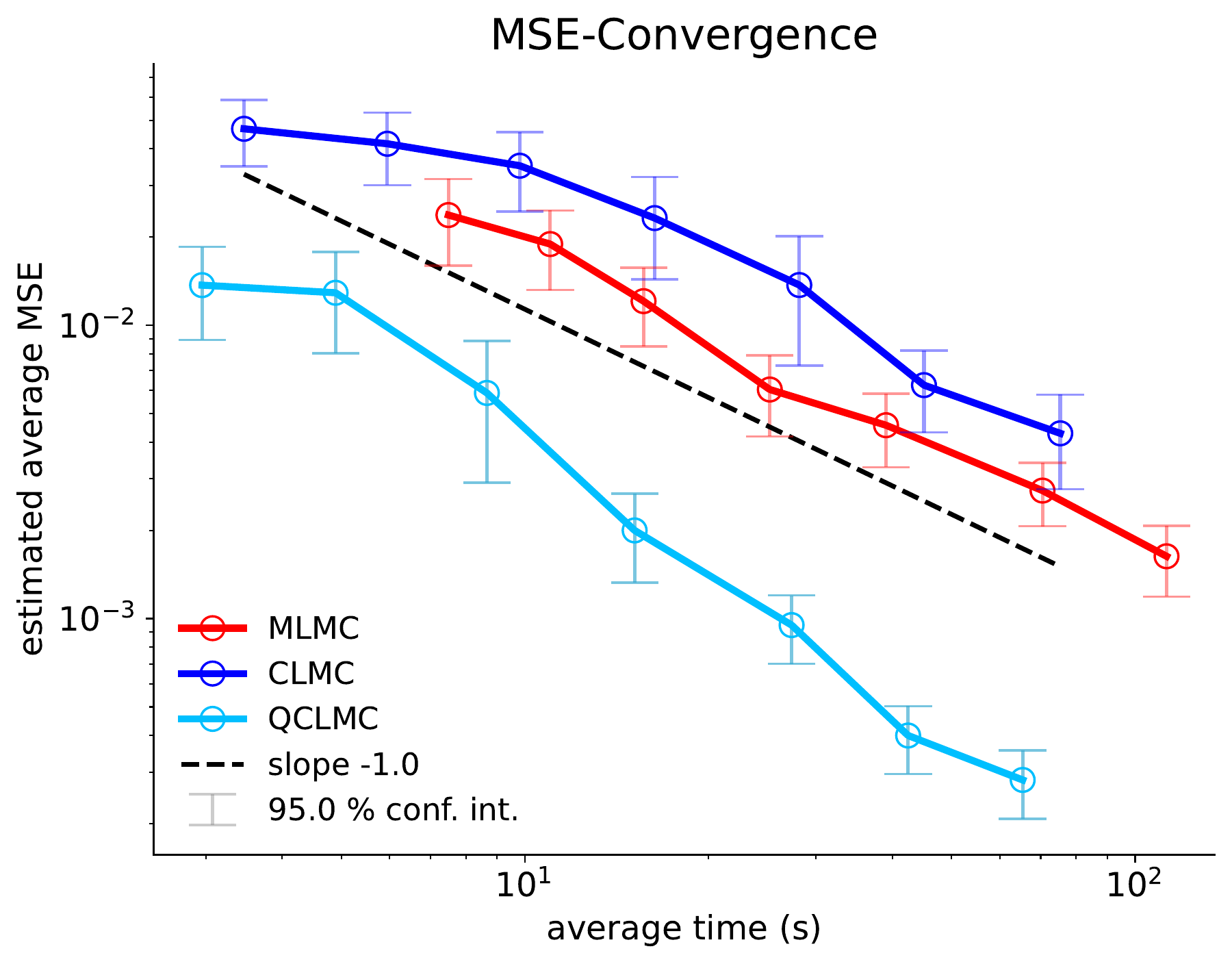}
  \includegraphics[width=0.49\textwidth]{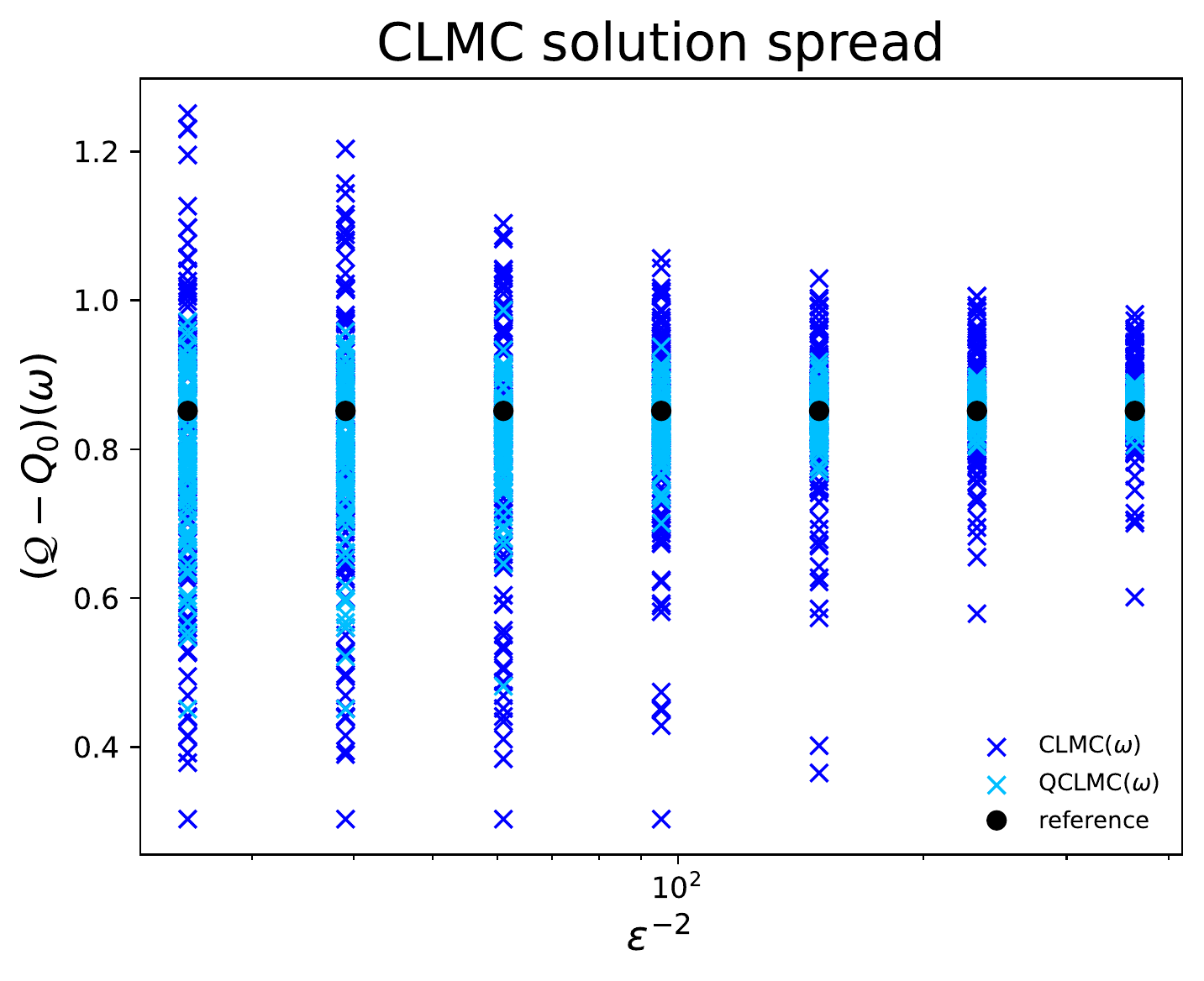}
   \includegraphics[width=0.49\textwidth]{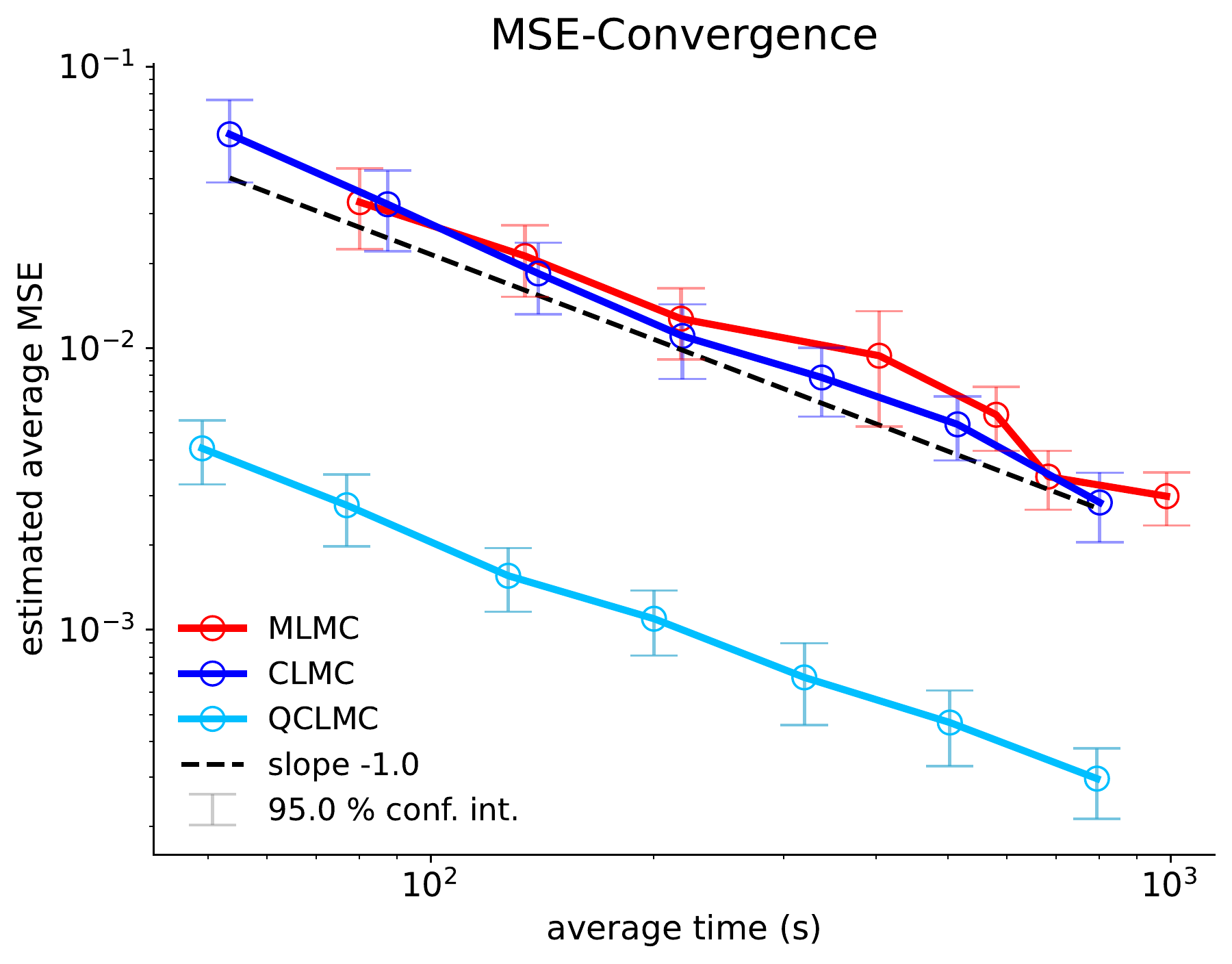}
  \includegraphics[width=0.49\textwidth]{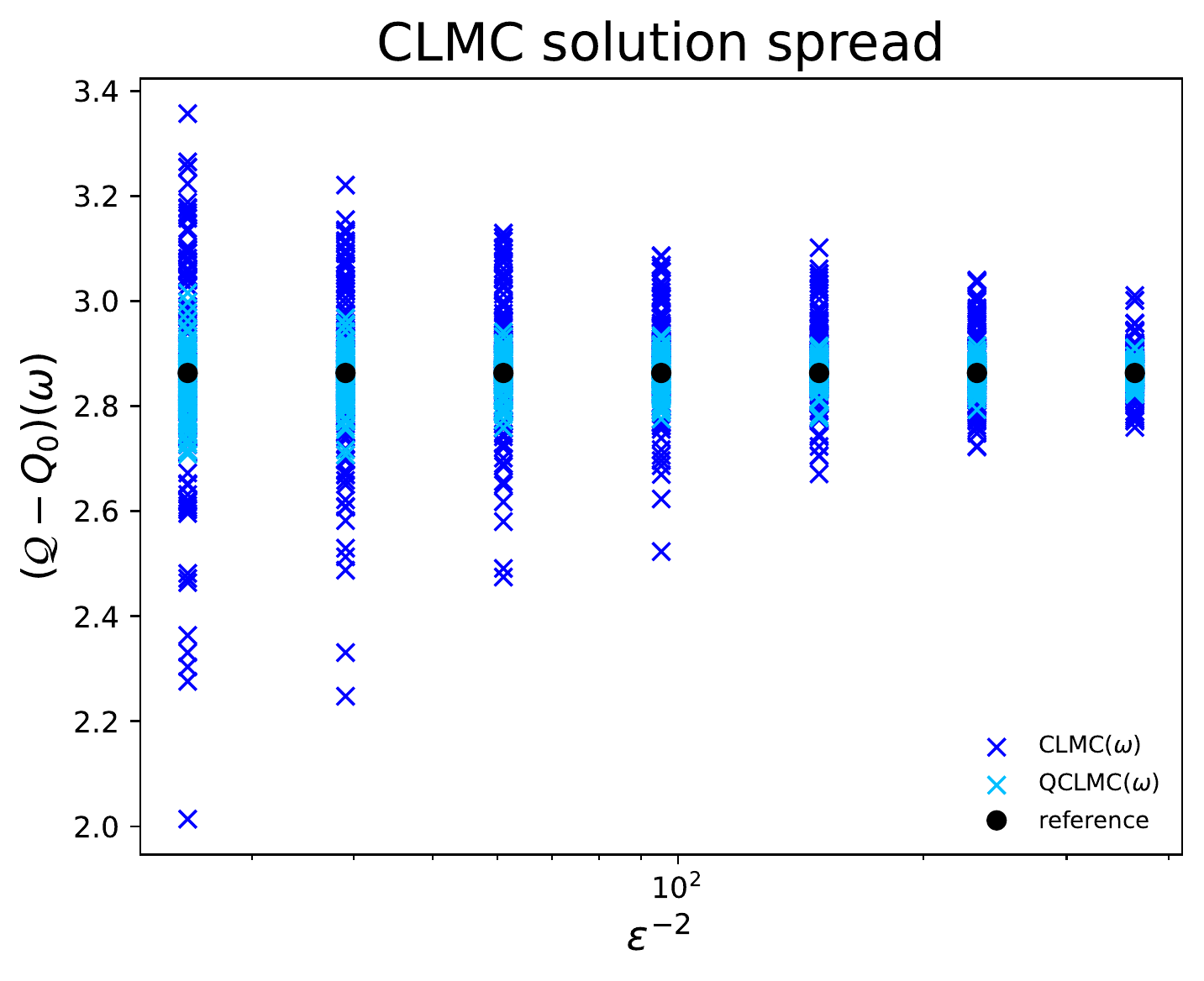}
 \caption{Left: The estimated average MSE (y-axis) plotted against time in seconds (x-axis) for the box coefficient, Example $1$ from Section \ref{BeBa_sec:random_pde_model}, with $P=300$ (top) and $P=1000$ (bottom). Right: The computed solutions (y-axis) for each of the $100$ runs plotted against the inverse value of the given MSE tolerances (x-axis) given in Equation \eqref{BeBa_eq:mse_tolerances}.}
 \label{BeBa_fig:convergence_and_solution_spread_box}
\end{figure}
\begin{figure}[tbhp]
 \includegraphics[width=0.49\textwidth]{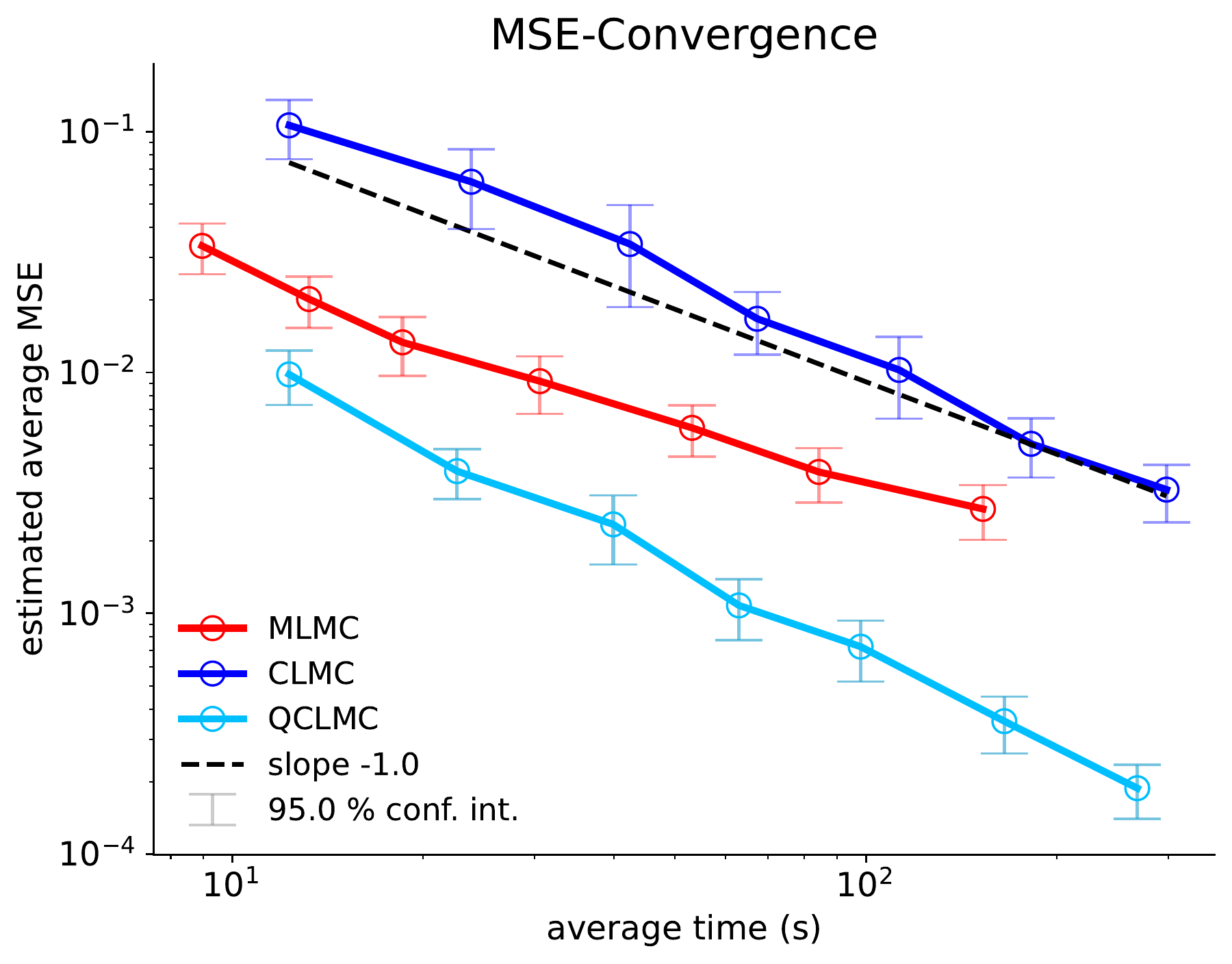}
 \includegraphics[width=0.49\textwidth]{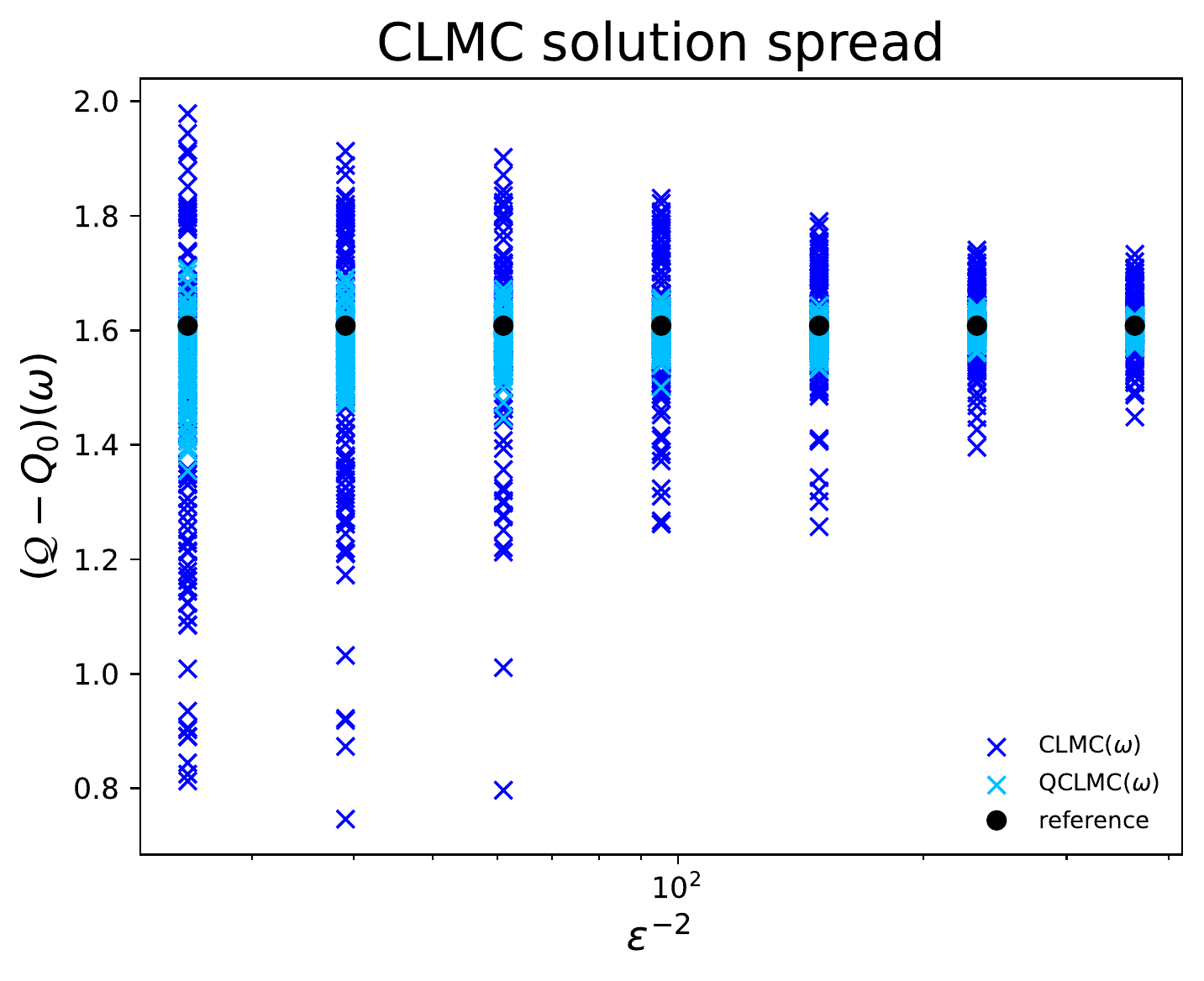}
 \includegraphics[width=0.49\textwidth]{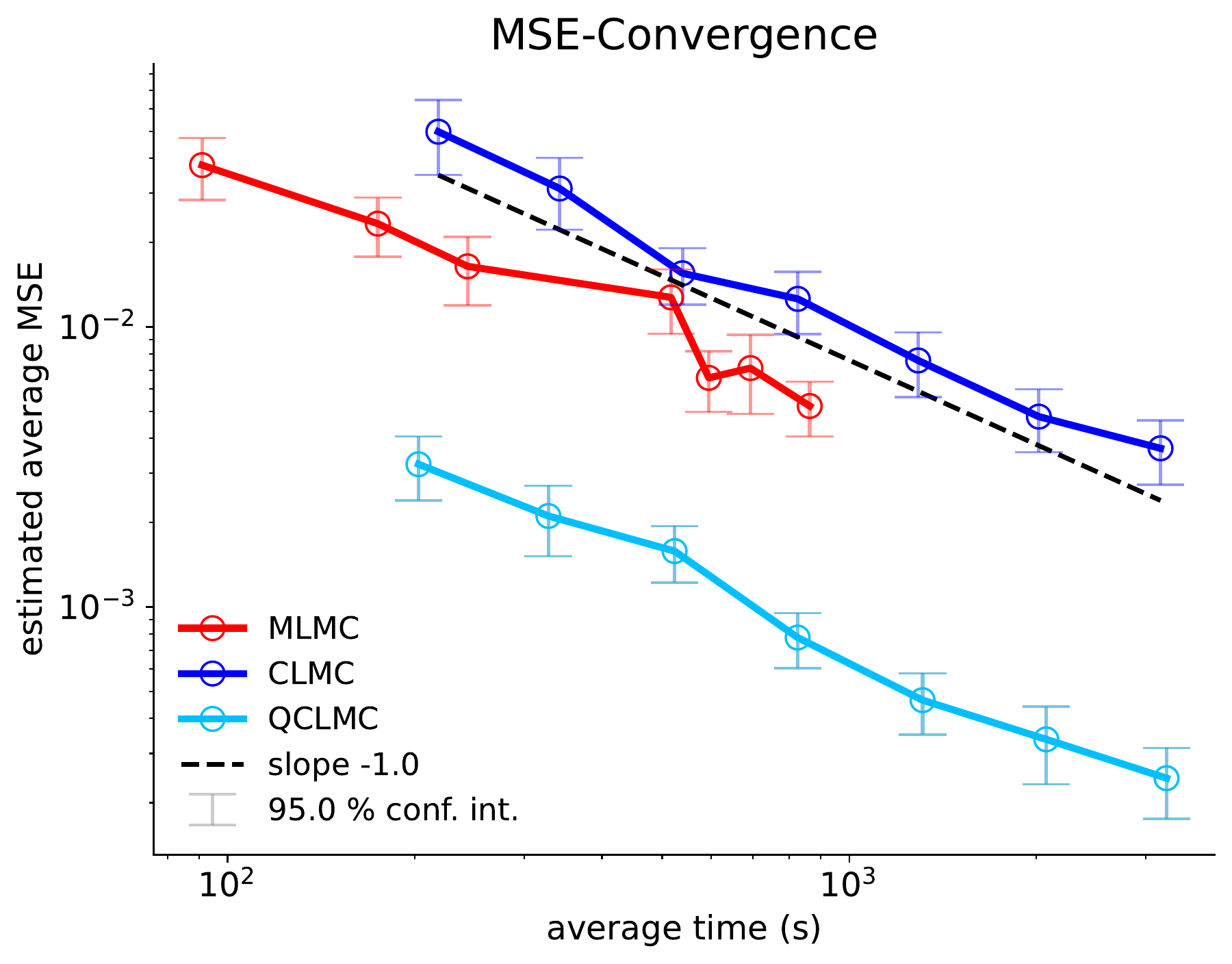}
 \includegraphics[width=0.49\textwidth]{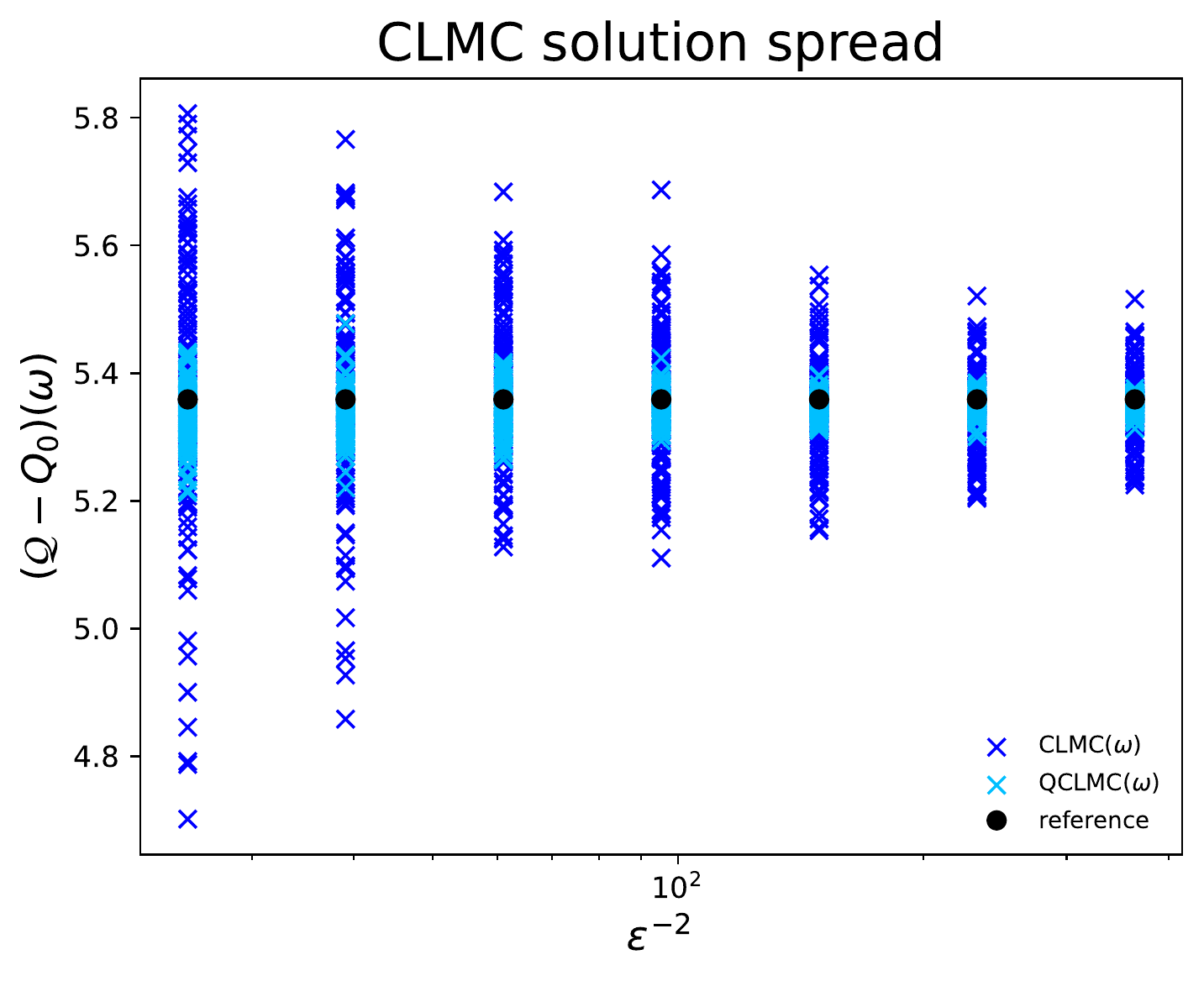}
 \caption{Left: The estimated average MSE (y-axis) plotted against time in seconds (x-axis) for the cross coefficient, Example $2$ from Section \ref{BeBa_sec:random_pde_model}, with $P=300$ (top) and $P=1000$ (bottom). Right: The computed solutions (y-axis) for each of the $100$ runs plotted against the inverse value of the given MSE tolerances (x-axis) given in Equation \eqref{BeBa_eq:mse_tolerances}. }
 \label{BeBa_fig:convergence_and_solution_spread_cross}
\end{figure}
\begin{remark}
\label{BeBa_rmk:independence}
 To obtain independent runs, different random seeds are used for each run. The exponentially distributed samples have their own seeding independent of the PDE coefficient samples. In standard CLMC the exponential numbers are drawn with the numpy library \cite{BeBa_Numpy2020} and in QCLMC the quasi-random numbers are generated by Sobol numbers \cite{BeBa_Sobol1967_Sobol} with Owen scrambling \cite{BeBa_Owen1995_Scrambling, BeBa_Owen1998_Scrambling} through the scipy library \cite{BeBa_Scipy2020}.
\end{remark}
\textbf{Acknowledgement:}
We thank the anonymous referees for constructive comments leading to a significant improvement of the manuscript and Robin Merkle for helpful discussions fostering ideas in particular on numerical aspects.
The work of Cedric Aaron Beschle is funded by the Deutsche Forschungsgemeinschaft (DFG, German Research Foundation) -- Project-ID 251654672 -- SFBTRR 161.

\end{document}